\documentclass{amsart}
\usepackage{amsmath,amssymb,amsthm,amsfonts,epsfig,verbatim,mathtools}
\usepackage[all]{xy}

\renewcommand{\bar}[1]{\overline{#1}}
\makeatletter
\@addtoreset{equation}{section}
\renewcommand{\theequation}{\thesection.\@arabic\c@equation}
\newcommand\Nopagebreak{\@nobreaktrue\nopagebreak}

\newcommand{\com}[1]{
}

\newcommand{\bm}{\begin{matrix}}
\newcommand{\enm}{\end{matrix}}
\newcommand{\bp}{\begin{pmatrix}}
\newcommand{\ep}{\end{pmatrix}}
\newcommand{\bsp}{\begin{split}}
\newcommand{\esp}{\end{split}}
\newcommand{\bc}{\begin{center}}
\newcommand{\ec}{\end{center}}

\newcommand{\ol}{\overline}
\newcommand{\liealgebra}[1]{{\mathfrak {#1}}}

\newcommand{\gl}{\liealgebra{gl}}

\newcommand{\liegroup}[1]{{\operatorname{#1}}}

\newcommand{\K}{\liegroup{K}}

\newcommand{\GL}{\liegroup{GL}}

\newcommand{\Un}{\liegroup{U}}

\newcommand{\Ad}{\liegroup{Ad}}
\DeclareMathOperator{\rank}{rank}
\renewcommand{\Re}{\operatorname{Re}}
\renewcommand{\Im}{\operatorname{Im}}
\newcommand{\im}{\operatorname{im}}
\newcommand{\eps}{\epsilon}

\newcommand{\R}{\mathbb R}

\newcommand{\C}{\mathbb C}

\newcommand{\Z}{\mathbb Z}
\newcommand{\N}{\mathbb N}

\newcommand{\mcl}{\mathcal L}

\newcommand{\cp}[1]{{\C \mathbb{P}^{#1}}}

\newcommand{\al}{\alpha}

\newcommand{\tr}{\rm tr}

\newcommand{\la}{\lambda}

\newcommand{\lglc}{{\mcl^- \GL_n \C}}
\newcommand{\lglr}{{\mcl^- \GL_n \R}}
\newcommand{\lupq}{{\mcl^- \Un_{p,q} }}
\newcommand{\lupp}{{\mcl^- \Un_{p,p} }}
\newcommand{\lun}{{\mcl^- \Un_{n} }}

\newcommand{\ii}{\mathrm i \, }

\def\Id{\operatorname{I}}

\newcommand{\bpm}{\begin{pmatrix}}
\newcommand{\epm}{\end{pmatrix}}
\usepackage[mathscr]{eucal}
\usepackage{bbm}
\usepackage{oldgerm}

\theoremstyle{plain}
\newtheorem{thm}{Theorem}[section]
\newtheorem{lem}[thm]{Lemma}

\newtheorem{prop}[thm]{Proposition}

\newtheorem{rem}[thm]{Remark}
\theoremstyle{definition}

\newcommand{\be}{\begin{equation}}
\newcommand{\ee}{\end{equation}}
\newcommand{\bs}{\begin{section}}
\newcommand{\es}{\end{section}}
\newcommand{\bss}{\begin{subsection}}
\newcommand{\ess}{\end{subsection}}
\newcommand{\bit}{\begin{itemize}}
\newcommand{\eit}{\end{itemize}}

\def\pr{\mathbb{P}}
\def\Aut{\operatorname{Aut}}
\def\diag{\operatorname{diag}}
\def \l {\lambda}
\def \a {\alpha}

\begin{document}
\title{Projective loops generate rational loop groups}

\author{Gang Wang}
\address{School of Mathematics, Shandong University, Jinan, Shandong Province, China}
\email{2017018@dgut.edu.cn}
\author{Oliver Goertsches}
\address{Fachbereich Mathematik und Informatik der Philipps-Universit\"at Marburg - Hans-Meerwein-Strasse 6, Marburg, Germany}
\email{goertsch@mathematik.uni-marburg.de}
\author{Erxiao Wang}
\address{Department of Mathematics, Hong Kong University of Science \& Technology, Clear Water Bay, Kowloon, Hong Kong} 
\email{maexwang@ust.hk} 


\begin{abstract}
Rational loops played a central role in Uhlenbeck's construction of harmonic maps into $\Un_n$ (chiral model in physics), and they are generated by simple elements with one pole and one zero constructed from Hermitian projections. It has been believed for long time that nilpotent loops should be added to generate rational loop groups with noncompact reality conditions. We prove a somewhat unexpected theorem that projective loops are enough to generate the rational loop groups $\lglc$, $\lglr$, and $\lupq$. 
\end{abstract}

\maketitle


\section{Introduction}


One motivation to study generator theorems for rational loop groups is clear: the dressing action of any rational loop on the solution space of an integrable PDE system or associated geometric problem can be computed as compositions of the actions of simple generators. Ideally the geometric transformations corresponding to these simple generators might also be easy to control. Thus a brief review of dressing and geometric transformations follows. 


The subject of geometric transformations, generating new solutions of geometric problems from any given one such as B\"acklund transformations, has a long history and was a central topic of classic differential geometry. The main idea came from the geometry of line or sphere congruence. In contrast, the original idea of dressing to generate new solutions of soliton equations or integrable systems was quite recent, and goes back to Zakharov-Shabat \cite{Zak79}. It looks simple in algebra: Given unique group factorizations via two subgroups in different orders $G=G_- \cdot G_+ = G_+ \cdot G_-$, e.g. $g_- g_+=\tilde{g}_+ \tilde{g}_-$. Then $g_- \ast g_+ := \tilde{g}_+$ defines an action of $G_-$ on $G_+$, named dressing. In our context $G_-$ and $G_+$ are usually two complementary subgroups of a loop group, for example $\mcl^- \GL_n \C$ (holomorphic near and normalized at $\infty$) and $\mcl^+ \GL_n \C$ (holomorphic on $\C$) are two complementary subgroups of $\mcl \GL_n \C$. Then Birkhoff Factorisation Theorem guarantees that the dressing action is well-defined at least locally.  

But the link between dressing and classical geometric transformations is not fully justified until the work of Terng-Uhlenbeck \cite{Ter00}. It is based on the dressing of a very simple rational loop in $\mcl^- \GL_n \C$: 
 \[
  p_{\al,\beta,V,W} = \left(\frac{\la-\al}{\la-\beta}\right)\pi_V +\pi_W \;  ,
  \]
  where $\C^n=V\oplus W$ and $\pi_V$, $\pi_W$ are projections onto $V$, $W$ along the other subspace. Then its dressing on  $h \in \mcl^+ \GL_n \C$ has a simple formula by residue calculus: $p_{\al,\beta,V,W} \ast h = p_{\al,\beta,V,W} \cdot h \cdot p_{\beta,\al,\tilde{V},\tilde{W}} \in \mcl^+ \GL_n \C$, where $\tilde{V}:=h(\alpha)^{-1} (V)$ and $\tilde{W}:=h(\beta)^{-1} (W)$, as long as $\tilde{V}\oplus \tilde{W}=\C^n$. In addition to many Darboux-B\"acklund transformations identified as dressing in  \cite{Ter00}, Ribaucour transformations in Terng-Wang \cite{Ter081}, Tzitz\'eica transformations in Wang \cite{Wan06} and Lin-Wang-Wang \cite{Lin152}, etc., have also been identified as dressing.   Dressing of rational loops also played an important role in Uhlenbeck's construction of harmonic maps into $\Un_n$ (or classical solutions of chiral model in physics) \cite{Uhl89}. It has been proved there that $\lun$ is generated by simple elements of the form
\[
g_{\alpha,\pi} = \left( \frac{\lambda-\alpha}{\lambda-\bar{\alpha}}\right)  \pi + (\Id-\pi),
\]
where $\pi$ is any Hermitian orthogonal projection of $\C^n$ onto some subspace.

Motivated by this work, Donaldson-Fox-Goertsches \cite{Neil08} found generators for the rational loop groups of all classical compact matrix groups and $G_2$. Then Goertsches \cite{Goe13} showed any element in the full rational loop group $\lglc$ can be written as  products of $p_{\al,\beta,V,W}$ and
  \[
  m_{\al,k,N} = \Id+\left(\frac{1}{\la-\al}\right)^k N \; ,
  \]
where $k$ is a positive integer and $N$ is a two-step nilpotent map, i.e. $N^2=0$. 

The new type of loops $m_{\al,k,N}$ has the special feature of having only one singularity and is constructed from nilpotent maps instead of projections. However their dressing formulas are quite complicated as showed in \cite{Goe13}.  We simply call $p_{\al,\beta,V,W}$ and $m_{\al,k,N}$  projective and nilpotent loops respectively. Goertsches also showed a similar result for $\lglr$. Such nilpotent loops are naturally added as generators from the point of view of Lie theory as an analogue of Iwasawa decomposition for noncompact real semisimple Lie groups. 

While our motivation in dressing or geometry is clear, our main theorems in this paper are still a little surprising from the point of view of Lie theory: some projective loops are enough to generate the rational loop groups $\lglc$, $\lglr$, and $\lupq$. Thus any rational dressing can be computed via iterations of projective ones, which can be carried out efficiently by computer software. 

In Section \ref{sec:GL} we first review Goertsches' theorems and summarize his technical induction involving nilpotent loops as a theorem: The polynomial loops $\GL_n \C [ ( \lambda - \alpha )^{-1} ]$, or the rational loops with only one fixed singularity are generated by nilpotent loops with the same pole. Then we prove our first two theorems that these nilpotent loops are in fact products of projective loops, which implies $\lglc$ and $\lglr$ are generated by projective loops alone despite their non-compactness.

The rest of the paper is devoted to the extension of these results to $\lupq$. In Section \ref{Sec:faforNUpq} we show the rational loops in $\lupp$ with only one fixed singularity are generated by similar nilpotent loops $m_{\al,k,N}$, while the rational loops in $\lupq$ ($p\ne q$) with only one fixed singularity are generated by $m_{\al,k,N}$ and a new type of nilpotent loops $n_{\al,k,N}$ where $N$ is $3$-step nilpotent. In Section \ref{sec:pmnGUpq} we apply standard induction to give a table of generators for $\lupq$ with both projective and nilpotent loops.  Finally in Section \ref{sec:pGUpq} we show that these nilpotent loops as generators turns out also to be products of projective loops in $\lupq$. While the proofs seem easy for $\lglc$, $\lglr$ cases from the basic example,  the proof for $\lupq$ case turns out to be quite involved. 

In the last section we propose a few open problems. For example, it would be desirable to have a direct proof of our generating theorems without taking the detour through the ``unnecessary'' nilpotent loops.

\bigskip 

\noindent
\textbf{Preliminaries and Notation:}\label{PTNA}

For convenience, this paper uses $\lglc$ to denote the full rational loop group, i.e., the set of all non-degenerate $n \times n$ matrices of rational functions, normalized at $\infty$: $g(\infty)=\Id$. Let  $\lglr$ denote its subgroup satisfying the $\GL_n \R$-reality condition: $ \tau( g(\bar{\lambda}) )= g(\lambda)$, where $\tau(A)=\bar{A}$ is the antiholomorphic involution fixing $\GL_n \R$. For real parameters, such rational loops naturally take values in $\GL_n \R$.

Assume that $0<p\le q$, $p+q=n$. The group $\Un_{p,q}$ is the fixed point set of the antiholomorphic involution $\tau(A)=(A^*)^{-1}$, where $A^*$ is the adjoint of $A$ with respect to the inner product $\langle v,w\rangle=-\sum_{i=1}^p \ol{v_i}{w_i}+\sum_{i=p+1}^{p+q} \ol{v_i}{w_i}$. Denoting $s=\diag(-\Id_p, \Id_q)$ we have $A^*=s \bar{A}^t s$. Then $\lupq$ denotes the subgroup of $\lglc$ satisfying the $\Un_{p,q}$-reality condition:
$ \tau( g(\bar{\lambda}) )= g(\lambda)$. For real parameters these loops take values in  $\Un_{p,q}$.

We say that $\alpha\in \cp{1}$ is a pole of $g\in \lglc$ if $\alpha$ is a pole of some matrix entry of $g$. If $\alpha$ is not a pole of g, we say that $\alpha$ is a zero of $g$ if $g(\alpha)$  is singular. Finally, $\alpha$ is a singularity of $g$ if it is a pole or a zero.
If $\alpha\in \cp{1}$ is a pole of $g$, there is a unique number $k\ge1$ such that the map $(\lambda-\alpha)^{k-1} g$ has a pole at $\alpha$, but $(\lambda-\alpha)^k g$ has no pole at $\alpha$. If we denote the evaluation of $(\lambda-\alpha)^k g$ at $\alpha$ by $A$, we call the pair $(k,\rank A)$ the \textbf{pole data} of $g$ at $\alpha$. There is a natural ordering on the possible pole data: $(k_1, n_1) < (k_2, n_2)$ if and only if $k_1 < k_2$ or ($k_1 = k_2$ and $n_1 < n_2$). It thus makes sense to compare degrees of poles.

\section{Factorization of nilpotent loops}\label{sec:GL}

In \cite{Goe13} the second named author obtained generating theorems for the full rational loop group of $\GL(n,\C)$ and for the subgroup of loops satisfying the reality condition with respect to the noncompact real form $\GL(n,\R)$. The new feature of these theorems, compared to other generating theorems, was the occurrence of a new type of simple loops $m_{\al,k,N}$ with only one singularity, see Table \ref{tab:GLnC} below.

For the full rational loop group, the theorem read
\begin{thm}[{\cite[Theorem 3.1]{Goe13}}]\label{thm:oldfullgenthm}
The rational loop group $\lglc$ is generated by the simple elements given in Table \ref{tab:GLnC}.
\end{thm}

\begin{table}[h]
\centering
\begin{tabular}{|c|c|c|}
\hline
Name & Definition & Conditions\\
\hline
$p_{\al,\beta,V,W}$ & $\left(\frac{\la-\al}{\la-\beta}\right)\pi_V +\pi_W$ & $\alpha\neq \beta\in \C,\, \C^n=V\oplus W$ \\
\hline
$m_{\al,k,N}$ & $\Id+\left(\frac{1}{\la-\al}\right)^k N$ & $\alpha\in \C,\, N:\C^n\to \C^n,\, N^2=0$\\
\hline
\end{tabular}\\[0.5cm]
\caption{Generators for $\lglc$}
\label{tab:GLnC}
\end{table}

The hardest part in the induction proof of this theorem is to deal with the special loop with only one singularity. We summarize the algebraic property of nilpotent loops as a separate theorem:

\begin{thm}[]
The rational loops in $\lglc$ and $\lglr$ with only one fixed singularity at $\a$, or the polynomial loops $\GL_n  [ ( \lambda - \alpha )^{-1} ]$:
\[
g(\la)=(\la-\al)^{-r}A_r+\ldots+(\la-\al)^{-1}A_1+I,
\]
are generated by the nilpotent loops $m_{\a,k,N}$ in Tables \ref{tab:GLnC} respectively \ref{tab:glnr}.
\end{thm}

In this section we aim to show that the nilpotent loops $m_{\al,k,N}$ are in fact not necessary to generate this loop group. Namely, we show:
\begin{thm}\label{sec2thm1} The rational loop group $\lglc$ is generated by the simple elements $p_{\alpha,\beta,V,W}$ alone.
\end{thm}

\begin{proof}
Using Theorem \ref{thm:oldfullgenthm} we only need to show that any nilpotent loop $m_{\al,k,N}$ can be written as a product of simple elements of type $p_{\alpha,\beta,V,W}$. We do this by induction on $k$.

Any two-step nilpotent endomorphism $N$ has a Jordan normal form $SNS^{-1}=J$ consisting only of Jordan blocks $\bpm 0\epm$ and $\bpm 0 & 1 \\0 & 0 \epm$. As then $Sm_{\a,k,N}(\l)S^{-1}=m_{\a,k,J}(\l)$, we only have to prove that $m_{\a,k,J}(\l)=I+\frac{1}{(\l-\a)^k}J$ can be factored as the product of projective loops, and for this it suffices to verify it for the case that $J=\bpm 0 &1 \\0 & 0 \epm $.

For $k=1$, we choose an auxiliary $\beta\neq\a$, write
 \begin{eqnarray*}
 m_{\a,1,J}(\l)=\bpm 1 & \frac{1}{\l-\a} \\0 & 1 \epm=\bpm \frac{\l-\beta}{\l-\a} & 0 \\0 & 1 \epm\bpm\frac{\l-\a}{\l-\beta}& \frac{1}{\l-\beta} \\0 & 1 \epm.
 \end{eqnarray*}
and observe that
 \begin{eqnarray*}
 \bpm \frac{\l-\beta}{\l-\a} & 0 \\0 & 1 \epm \text{ and } \bpm\frac{\l-\a}{\l-\beta}& \frac{1}{\l-\beta} \\0 & 1 \epm=\frac{\l-\a}{\l-\beta} \bpm 1& \frac{1}{\beta-\a}\\0 & 0\epm+\bpm 0& -\frac{1}{\beta-\a}\\0 & 1\epm
 \end{eqnarray*}
 are both projective loops. Assuming the result holds for $k=1,2,\ldots ,n-1$, we now prove it for $k=n$. Choosing again any $\beta\neq\a$, we compute
 \begin{align*}
  &m_{\a,n,J}(\l)=\bpm 1 & \frac{1}{(\l-\a)^n} \\0 & 1 \epm \notag\\
                  &\,\,=\bpm \frac{(\l-\beta)^n}{(\l-\a)^n} &0 \\0 & 1 \epm\bpm  \frac{(\l-\a)^n}{(\l-\beta)^n} & \frac{1}{(\l-\beta)^n} \\0 & 1 \epm\notag\\
                  &\,\,=\bpm \frac{(\l-\beta)^n}{(\l-\a)^n} &0 \\0 & 1 \epm\left[\frac{(\l-\a)^n}{(\l-\beta)^n}\bpm 1 &\frac{1}{(\l-\a)^n-(\l-\beta)^n} \\0 & 0 \epm+\bpm 0 &-\frac{1}{(\l-\a)^n-(\l-\beta)^n} \\0 & 1\epm\right]\notag\\
                  &\,\,=\bpm \frac{\l-\beta}{\l-\a} &0 \\0 & 1 \epm^{n}\left[\frac{\l-\a}{\l-\beta}\bpm 1 &\frac{1}{(\l-\a)^n-(\l-\beta)^n} \\0 & 0 \epm+\bpm 0 &-\frac{1}{(\l-\a)^n-(\l-\beta)^n} \\0 & 1\epm\right]^n.\notag
 \end{align*}
 The expression inside the square brackets factors as
\[ \bpm  \frac{\l-\a}{\l-\beta} &\frac{(\beta-\a)}{(\l-\beta)[(\l-\a)^n-(\l-\beta)^n]} \\0 & 1\epm =\bpm  \frac{\l-\a}{\l-\beta} &0 \\0 & 1\epm \bpm 1 &\frac{(\beta-\a)}{(\l-a)[(\l-\a)^n-(\l-\beta)^n]} \\0 & 1\epm
\]
so we just have to show that
 \begin{eqnarray} \label{nilfpr}
 \bpm 1 &\frac{(\beta-\a)}{(\l-\a)[(\l-\a)^n-(\l-\beta)^n]} \\0 & 1\epm
 \end{eqnarray}
can be factored as the product of projective loops. Because $\a$ is only a simple root of the degree $n$ polynomial $(\l-\a)[(\l-\a)^n-(\l-\beta)^n]$, we can use partial fraction decomposition to decompose \eqref{nilfpr} as the product of factors of the form $ \bpm 1 &\frac{c_{i}}{(\l-\beta_{i})^{l_{i}}} \\0 & 1\epm$, where $0<l_{i}<n$, and apply the induction hypothesis, to each factor.
\end{proof}

The generators for $\lglr$ were found similarly in \cite{Goe13}:
 \begin{thm}[{\cite[Theorem 4.1]{Goe13}}]\label{sec2thm3}
 The rational loop group $\lglr$  is generated by the simple elements in Table \ref{tab:glnr}.
 \end{thm}
\begin{table}[h]
\begin{tabular}{|c|c|c|}
\hline
Name & Definition & Conditions\\
\hline
$p_{\al,\beta,V,W}$ & $\left(\frac{\la-\al}{\la-\beta}\right)\pi_V +\pi_W$ & \begin{tabular}{c} $\alpha,\beta\in \R,\, \C^n=V\oplus W$ \\ $\bar{V}=V,\, \bar{W}=W$ \end{tabular} \\
\hline
$q_{\alpha,\beta,V,W}$ & $\frac{(\la-\alpha)(\la-\bar{\al})}{(\la-\beta)(\la-\bar{\beta})}\pi_V+\pi_W$ & \begin{tabular}{c} $\al$ or $\beta\notin \R,\, \C^n=V\oplus W$\\ $\bar{V}=V,\, \bar{W}=W$\end{tabular}\\
\hline
$r_{\alpha,\beta,V,W}$ & $\left(\frac{\la-\alpha}{\la-\beta}\right)\pi_V + \pi_W + \left(\frac{\la-\bar{\alpha}}{\la-\bar{\beta}}\right) \pi_{\bar{V}}$ & \begin{tabular}{c} $\C^n=V\oplus W\oplus \bar{V}$ \\ $V\cap \bar{V}=0,\, \bar{W}=W$ \end{tabular} \\
\hline
$m_{\al,k,N}$ & $\Id+\left(\frac{1}{\la-\al}\right)^k N$ & $\al\in\R,\, N^2=0,\, \bar{N}=N$\\
\hline
\end{tabular}\\[0.5cm]
\caption{Generators for $\lglr$}
\label{tab:glnr}
\end{table}
Similarly to the case of the full rational loop group, we can show that the nilpotent loops $m_{\alpha,k,N}$ are not necessary and obtain:
\begin{thm}\label{sec2thm4}
The rational $\lglr$  is generated by the simple elements $p_{\alpha,\beta,V,W}$, $q_{\alpha,\beta,V,W}$ and $r_{\alpha,\beta,V,W}$ alone.
\end{thm}
\begin{proof}
The proof is similar to the proof of Theorem \ref{sec2thm1}; we show that any loop of the form $m_{\alpha,k,N}$ is a product of loops $p_{\alpha,\beta,V,W}$ and $q_{\alpha,\beta,V,W}$, by induction on $k$.

A real two-step nilpotent matrix is, in $\GL(n,\R)$, conjugate to a block diagonal matrix, with blocks only $\bpm 0\epm$ and $J:=\bpm 0 & 1 \\0 & 0 \epm$. So again, we only have to show that the loops
\begin{equation}\label{eqn:nilfpr2}
 m_{\a,k,J}(\l)=\bpm 1 & \frac{1}{(\l-\a)^k} \\0 & 1 \epm,
\end{equation}
where $\alpha\in \R$, can be factored as the product of loops of the form $p$ and $q$. We can follow the proof of Theorem \ref{sec2thm1} (taking $\beta\in \R$) until the point where, in the induction step, we have to show that the loop
\[
\bpm 1 &\frac{(\beta-\a)}{(\l-\a)[(\l-\a)^n-(\l-\beta)^n]} \\0 & 1\epm,
\]
where $\alpha\neq\beta\in \R$, can be factored as a product of loops $p$ and $q$.

We first claim that the real polynomial $(\l-\a)[(\l-\a)^n-(\l-\beta)^n]$ has only simple (complex) roots. Clearly, $\alpha$ is a simple root. If $\mu\in \C$ was a multiple root of this polynomial, then it was a zero of both $(\l-\a)^n-(\l-\beta)^n$ and its derivative $n(\l-\a)^{n-1}-n(\l-\beta)^{n-1}$. But then
\[
\left(\frac{\mu-\beta}{\mu-\alpha}\right)^n = 1 = \left(\frac{\mu-\beta}{\mu-\alpha}\right)^{n-1},
\]
which implies $\mu-\alpha = \mu-\beta$, i.e., which was only possible if $\alpha=\beta$.

Thus, partial fraction decomposition shows that the loop \eqref{eqn:nilfpr2} can be factored as the product of loops of the form
 \begin{enumerate}
 \item  $\bpm 1 & \frac{d}{\l-\zeta} \\0 & 1 \epm$, $\zeta\in \R$, $d\in\R$,
 \item  $\bpm 1 & \frac{e}{(\l-\xi)(\l-\bar\xi)} \\0 & 1 \epm$, $\xi\in\C\setminus\R$, $e\in\R$ and
 \item $\bpm 1 & \frac{f\l}{(\l-\xi)(\l-\bar\xi)} \\0 & 1 \epm$, $\xi\in\C\setminus\R$, $f\in\R$.
\end{enumerate}
Loops of the first kind were already dealt with in the base case of the induction. For the factorization of the second type of loops let $\gamma$, $\bar \gamma\in\C\setminus\R$ be the conjugate roots of $(\l-\xi)(\l-\bar\xi)+1=0$, so that
\[
(\lambda-\gamma)(\lambda-\bar\gamma) = (\lambda-\xi)(\lambda-\bar\xi) + 1.
\]
Then we can decompose
 \begin{align*}
 \bpm 1 & \frac{e}{(\l-\xi)(\l-\bar\xi)} \\0 & 1 \epm &=\bpm \frac{(\l-\gamma)(\l-\bar\gamma)}{(\l-\xi)(\l-\bar\xi)}& 0 \\0 & 1 \epm\bpm \frac{(\l-\xi)(\l-\bar\xi)}{(\l-\gamma)(\l-\bar\gamma)} & \frac{e}{(\l-\gamma)(\l-\bar\gamma)} \\0 & 1 \epm\\
 &=\bpm \frac{(\l-\gamma)(\l-\bar\gamma)}{(\l-\xi)(\l-\bar\xi)}& 0 \\0 & 1 \epm\left[\frac{(\l-\xi)(\l-\bar\xi)}{(\l-\gamma)(\l-\bar\gamma)}\bpm 1 & -e \\0 & 0 \epm+\bpm 0 & e \\0 & 1 \epm\right]
 \end{align*}
 into a product of two loops of type $q$.

For the factorization of loops of the third kind we observe first that for a small positive number $\eta>0$ the polynomial  $(\l-\xi)(\l-\bar\xi)+\eta\l=0$ has no real root. Let $\gamma$, $\bar \gamma\in\C\setminus\R$ be the roots of $(\l-\xi)(\l-\bar\xi)+\eta\l=0$, i.e.,
\[
(\lambda-\gamma)(\lambda-\bar\gamma) = (\lambda-\xi)(\lambda-\bar\xi) + \eta\lambda.
\]
Then we can decompose
 \begin{align*}
 \bpm 1 & \frac{f\l}{(\l-\xi)(\l-\bar\xi)} \\0 & 1 \epm
 &=\bpm \frac{(\l-\gamma)(\l-\bar\gamma)}{(\l-\xi)(\l-\bar\xi)}& 0 \\0 & 1 \epm\bpm \frac{(\l-\xi)(\l-\bar\xi)}{(\l-\gamma)(\l-\bar\gamma)} & \frac{f\l}{(\l-\gamma)(\l-\bar\gamma)} \\0 & 1 \epm\\
 &=\bpm \frac{(\l-\gamma)(\l-\bar\gamma)}{(\l-\xi)(\l-\bar\xi)}& 0 \\0 & 1 \epm\left[\frac{(\l-\xi)(\l-\bar\xi)}{(\l-\gamma)(\l-\bar\gamma)}\bpm 1 & -\frac{f}{\eta}\\0 & 0 \epm+\bpm 0 & \frac{f}{\eta} \\0 & 1 \epm\right]
 \end{align*}
into a product of loops of type $q$.
\end{proof}

\section{The algebraic loops of $\lupq$ with only one singularity }\label{Sec:faforNUpq}
We will first study the negative algebraic loops of $\lupq$ with only one singularity, and prove that they can also be written as the products of some nilpotent loops. \\

\indent The nilpotent loops needed for the proof are given in Table \ref{tab:upqn:pneqq}, where the loops $n_{\al,k,N}$ are needed only for the case $p\neq q$.\\


\begin{table}[h]
\caption{Nilpotent loops for $\lupq$}
\begin{tabular}{|c|c|c|}
\hline
Name & Definition & Conditions\\
\hline
$m_{\al,k,N}$ & $\Id+\left(\frac{1}{\la-\al}\right)^k N$ & \begin{tabular}{c} $\al\in\R$,  $k\in\N$, \\ $N^2=0$,
 $N^*=-N$, $\rank N=1$ \end{tabular}\\
\hline
\begin{tabular}{c} $n_{\al,k,N}$ \\ $(p \neq q)$ \end{tabular} & $\Id+\left(\frac{1}{\la-\al}\right)^k N+\frac{1}{2}\left(\frac{1}{\la-\al}\right)^{2k} N^2$ & \begin{tabular}{c} $\al\in \R$, $k\in\N$, $N=M-M^*$, where\\ $M:(V\oplus sV)^\perp\to V$,\\ $V$ max.~isotropic,\\$\rank M=1$, $\rank N=2$ \end{tabular}\\
\hline
\end{tabular}
\label{tab:upqn:pneqq}
\end{table}
\indent Recall the notation for $\lupq$ defined in Section  \ref{PTNA}. Assume $V\subset \C^n$ is a subspace isotropic with respect to $\langle \cdot,\cdot\rangle$, then $\langle \cdot,\cdot \rangle$ is definite on $V\oplus sV$ and we have an Hermitian orthogonal decomposition of $\C^n$ into three subspaces: $\C^n=V\oplus (V\oplus sV)^\perp \oplus sV$. The last row in Table \ref{tab:upqn:pneqq} needs a little explanation: If $V$ is an isotropic subspace, and $M:(V\oplus sV)^\perp\to V$ an arbitrary linear map, then we denote also by $M$ the extension of $M$ by zero on $V\oplus sV$. Then $M^*$ satisfies $V^\perp\subset \ker M^*$ and $M^*(sV)\subset (V\oplus sV)^\perp$. It follows that $N=M-M^*$ sends $sV$ to $(V\oplus sV)^\perp$, which in turn is sent to $V$. In particular, $N^2$ sends $sV$ to $V$ and $N^3=0$.

To show that the nilpotent loops in Table \ref{tab:upqn:pneqq} satisfy the reality condition observe that they fit into the following framework:
\begin{lem} \label{lem:frame}
Let $N^*=-N$ and $N^r=0$ for some $r\ge 1$. Then for any $\al\in \R$ and $k\in \N$,
\[
g(\la)=\exp((\la-\al)^{-k} N)=\sum_{j=0}^{r-1} \frac{1}{j!} (\la-\al)^{-jk}N^{j}
\]
is a rational loop satisfying the $\Un_{p,q}$-reality condition.
\end{lem}
\begin{proof} Rationality is clear because of the nilpotency of $N$. For the reality condition we calculate
\[
\tau(g(\bar{\la}))=\exp(d\tau((\bar{\la}-\al)^{-k}N))=\exp(-(\la-\al)^{-k}N^*)=g(\la),
\]
where we used $d\tau(X)=-X^*$ for all $X\in \gl(n,\C)$.

\end{proof}
\begin{rem}
By the above lemma, it is easy to see $m_{\al,k,N}$ and $n_{\al,k,N}$ in Table \ref{tab:upqn:pneqq} satisfy the $\Un_{p,q}$-reality condition for any rank of $N$.
\end{rem}
The following lemmas about the existence of simple factors of type $m_{\al,k,N}$ and $n_{\al,k,N}$ will be crucial for the proof.
\begin{lem}\label{lem:Nexist} Let $N^*=-N$ and $N^2=0$. Then $\im N$ is isotropic, and $\ker N=(\im N)^\perp$. Conversely, for any $V\subset \C^n$ isotropic, and vectors $v\in V$ and $w\notin V^\perp$ such that $\langle v,w\rangle\in \ii\cdot \R$, there exists an anti-self-adjoint at most rank 2 endomorphism $N$ of $\C^n$ with $N^2=0$ satisfying $N(V^\perp)=0$, $\im N\subset V$ and $N(w)=v$.
\end{lem}
\begin{proof}
Let $N:\C^n\to \C^n$ be a two-step nilpotent anti-self-adjoint endomorphism. Since
\[
0=\langle N^2v,w\rangle =-\langle Nv,Nw\rangle,
\]
for all $v$ and $w$, the image of $N$ is isotropic. Furthermore $\ker N=(\im N^*)^\perp=(\im N)^\perp$.
 If we set $V=\im N$, then the map $N$ is defined by its restriction $N:sV\to V$, and the map $N\circ s:V\to V$ is anti-self-adjoint with respect to the standard inner product $(v,w)=\langle sv,w\rangle$:
\[
(Ns(v),w)=\langle sNsv,w\rangle=-\langle sv,Ns(w)\rangle =-(v,Ns(w)).
\]

So for the converse direction, if $v\in V$ and $w\notin V^\perp$ with $\langle v,w\rangle\in \ii\cdot \R$ are given, write $w=s\tilde{w}_0 + \tilde{w}$ with $\tilde w_0\in V$ and $\tilde{w}\in V^\perp$ and note that $(\tilde{w}_0,v)=\langle w,v\rangle \in \ii\cdot \R$. We can therefore choose a map $\tilde{N}:V\to V$, anti-self-adjoint with respect to $(\cdot,\cdot)$, that satisfies $\tilde{N}(\tilde{w}_0)=v$. Then $N:\C^n\to \C^n$, defined by $N(V^\perp)=0$ and $\left.N\right|_{sV}=\tilde{N}s$, is anti-self-adjoint with respect to $\langle \cdot,\cdot\rangle$ and $N(w)=N(s\tilde{w}_0)=\tilde{N}(\tilde{w}_0)=v$.\\
\indent Next we will explain that $N$ can be chosen to have rank at most $2$. For that, it suffices to choose $\tilde N$ with rank at most $2$. Let $m=\dim V$, and choose $U\in \Un_m$ such that $U(\tilde w_0)=\bpm 1&0&\ldots&0\epm^t$, $U(v)=\bpm z_1&z_2&\ldots&z_m\epm^t$. Since $(\tilde{w}_0,v)\in \ii\cdot \R$, then $z_1\in \ii\cdot \R$. Writing $\hat N=\bpm z_1& -\bar Z_2^t\\ Z_2& 0\epm$, where $Z_2=\bpm z_2&z_3&\ldots&z_m\epm^t$, we can define $\tilde N=U^*\hat N U$, which has rank at most 2.

\end{proof}

\begin{lem} \label{lem:N2exist} Assume that $p\neq q$ and let $V$ be maximal isotropic. If $u,v,w$ are vectors such that $w\in V$, $v\in V^\perp\setminus V$ and $u\notin V^\perp$, satisfying $\langle u,w\rangle\in \R\setminus \{0\}$ and
\[
2\cdot \langle u,w\rangle + \langle v,v\rangle=0,
\]
then there exists a rank 1 linear map $M:(V\oplus sV)^\perp \to V$ such that the anti-self-adjoint rank 2 endomorphism $N=M-M^*$ satisfies
\[
\frac{1}{2}N^2(u)+N(v)+w=0.
\]
\end{lem}
\begin{proof}
Let $M:\C^n\to \C^n$ be defined by $M=0$ on $V\oplus sV\oplus (v^\perp\cap (V\oplus sV)^\perp)$ and $M(v)=-2w$, and define $N:=M-M^*$. We have
\[
\langle N(u),v\rangle =-\langle u,N(v)\rangle=2\cdot \langle u,w\rangle
\]
and consequently,
\[
N(u)\in 2 \frac{\langle u,w\rangle}{\langle v,v\rangle}\cdot v + v^\perp\cap (V\oplus sV)^\perp+V=-v+v^\perp\cap (V\oplus sV)^\perp+V.
\]
It follows
\[
\frac{1}{2}N^2(u)+N(v)+w=w-2w+w=0
\]
as desired.
Next we consider the ranks of $M$ and $N$. Assume $v=v_0+v_1$ for $v_0\in V$, $0\neq v_1\in (V\oplus sV)^\perp$. Then we have
\[
V\oplus sV\oplus (v^\perp\cap (V\oplus sV)^\perp)=V\oplus sV\oplus (v_1^\perp\cap (V\oplus sV)^\perp)=v_1^\perp.
\]
It is easy to see $V\oplus sV\oplus (v^\perp\cap (V\oplus sV)^\perp)$ has dimension $n-1$, i.e. we have $\rank M=1$ and $\rank N=2$.
\end{proof}

Now we can discuss the factorizations of the negative algebraic loops with one singularity into nilpotent loops:
\begin{thm}\label{Upq:facnalstonil}
Assume $g(\l)$ is a negative algebraic loop of $\lupq$ with only singularity at $\a$, $\a\in\R$. Then when $p=q$ (resp.\ $p\neq q$), we have that $g(\l)$ can be written as the product of the nilpotent loops $m_{\al,k,N}$ (resp.\ $m_{\al,k,N}$, $n_{\al,k,N}$) as in Table \ref{tab:upqn:pneqq}.
\end{thm}
\begin{proof}
We write
\[
g(\la)=(\la-\al)^{-r}A_r+\ldots+(\la-\al)^{-1}A_1+A_0
\]
with $r\ge 1$, $A_r\neq 0$ and $A_0=\Id$. The reality condition written out explicitly is
\begin{equation}\label{eq:orthcond}
\sum_{i+j=k} \langle A_iv,A_jw\rangle=0
\end{equation}
for all $k\ge 1$.

The type of induction we will use is the same as in the second part of the proof of \cite[Theorem 3.1]{Goe13}. We use the same notation as there:
 $K_i=\bigcap_{j\ge i} \ker A_j$ for $i\ge 0$, so that
\[
\C^n=K_{r+1}\supset K_r\supset\ldots \supset K_1\supset K_0=0,
\] and $V_i:=\sum_{j\ge i} A_j(K_{j+1})$, so that
\[
0=V_{r+1}\subset V_r\subset\ldots\subset V_1\subset V_0.
\]
For $i\ge 1$, the spaces $V_i$ are isotropic, and $V_0$ is perpendicular to $V_1$, so since $g$ is supposed to be nonconstant, $V_0\neq \C^n$. Thus, no analogue of the last part of the proof of \cite[Theorem 3.1]{Goe13} is needed here.

Let ${\mathcal K}=\{ (a_i)_{i\ge 0}\mid a_i\in \N,\, \sum_i a_i=n\}$, equipped with the total ordering
\[
(a_i)_i< (b_i)_i\Longleftrightarrow \text{There exists } j\ge 0 \text{ such that } a_i=b_i \text{ for } i>j \text{ and } a_j < b_j.
\]
For a loop $g$ as above, we define an associated tuple $\eps(g)=(a_i)_i\in {\mathcal K}$ by $a_i:=\dim K_{i+1}-\dim K_i=\dim A_i(K_{i+1})$.

We will use Lemma \ref{lem:Nexist}, Lemma \ref{lem:N2exist} to construct nilpotent $m_{\a,k, N}$, $k\in\N$, $\rank N\leq 2$, $n_{\a,k, N}$, $k\in\N$, $\rank N=2$ to reduce the total ordering: $\eps(g)$. Since the remainder of the proof is significantly different in the case $p=q$, we treat it in the following two separate Propositions \ref{prop:p=q} and  \ref{prop:pneqq}. In addition, we will prove Proposition \ref{Prop:fN2tN1} to decompose $m_{\a,k, N}$, $k\in\N$, $\rank N=2$ into the product of two $m_{\a,k, N}$, $k\in\N$, $\rank N=1$. This will conclude the proof.
\end{proof}
\begin{prop} \label{prop:p=q} With the notation as above, in the case $p=q$ the loop $g$ is a product of simple factors of the form $m_{\al,k,N}$, $k\in\N$, $\rank N\leq 2$.
\end{prop}
\begin{proof}
We will prove the claim by induction on $\eps(g)$, the induction basis being trivial since the unique minimum is attained only for $g(\la)=\Id$.

Let $k\leq r$ be the smallest integer such that $\im A_k\subset V$, where $V$ is any maximal isotropic subspace containing $V_1$. Since $A_0=\Id$ and as noted above, $V_0\neq \C^n$ unless $g$ is the constant identity loop, we can assume $k\ge 1$.

We consider first the case that $k>1$. Then by definition of $V_{k-1}$, we have $A_{k-1}(K_k)\subset V_{k-1}\subset V$. Let $l> k$ be the smallest integer such that $A_{k-1}(K_l)\not\subset V$. Choose a vector $v\in K_l$ such that $A_{k-1}(v)\notin V$; then $v\notin K_{l-1}$ by definition of $l$. Since $l>k$, we  have $0\neq A_{l-1}(v)\in V_{l-1}\subset V_{k-1}\subset V$. Equation \eqref{eq:orthcond} gives
\begin{align}\label{eq:foro11}
0&=\sum_{i+j=k+l-2} \langle A_i v,A_j v\rangle\\
&= \langle A_{k-1}v,A_{l-1} v\rangle + \langle A_{l-1} v,A_{k-1}v\rangle=2\cdot \Re\,\langle A_{k-1}v,A_{l-1} v\rangle,\notag
\end{align}
since all the other summands vanish either because $v\in K_{l}$, or because $V$ is isotropic and for all indices $i,j\ge k$, the images $\im A_i$ and $\im A_j$ are in $V$.

We have $A_{l-1}v \in V$ (as argued above), as well as $A_{k-1}v \notin V=V^\perp$ (here we use $p=q$) and $\langle A_{k-1}v,A_{l-1}v\rangle \in \ii\R$. With the help of Lemma \ref{lem:Nexist} we may thus choose a two-step nilpotent, anti-self-adjoint  map $N$ of rank at most $2$ with $N(V)=0$ and $N(A_{k-1}v)=-A_{l-1} v$. We define
\begin{align*}
\tilde{g}(\la):&=m_{\al,l-k,N}(\la)g(\la)\\
&=\Id+\ldots + (\la-\al)^{-l+1} (NA_{k-1}+A_{l-1}) + (\la-\al)^{-l} A_{l} +\ldots + (\la-\al)^{-r} A_r\\
&=:\sum_i (\la-\al)^{-i}\tilde{A}_i,
\end{align*} and wish to show that $\eps(\tilde{g})<\eps(g)$ in order to use induction. To show this we have to investigate how $\tilde{K}_i=\bigcap_{j\ge i} \ker \tilde{A}_j$ has changed compared to $K_i$. Obviously, $K_i=\tilde{K}_i$ for $i\ge l$. If $w\in K_{l-1}$, we have
\[
(NA_{k-1}+A_{l-1})(w)=NA_{k-1}(w)\in NA_{k-1}(K_{l-1})\subset N(V)=0,
\]
so $K_{l-1}\subset \tilde{K}_{l-1}$, and furthermore, $(NA_{k-1}+A_{l-1})(v)=0$, so $\K_{l-1}\subsetneq\tilde{K}_{l-1}$, which means $\eps(\tilde{g})<\eps(g)$. This concludes the case $k>1$.

Let us assume now that $k=1$, i.e., that $\im A_1,\ldots,\im A_r\subset V$. Then \eqref{eq:orthcond} gives that
\begin{equation}\label{eq:orthcond2}
\langle A_iv,w\rangle + \langle v,A_iw\rangle = 0
\end{equation}
for all $i$ and all $v,w$, so that all matrices $A_i$ are skew-Hermitian with respect to the inner product $\langle \cdot,\cdot\rangle$. This implies that we can find a vector $v$ such that $\langle A_rv,v\rangle \neq 0$. (Indeed, $sA_r$ is skew-Hermitian with respect to the standard inner product, hence $isA_r$ is a nonzero Hermitian matrix, which necessarily has a nonzero eigenvalue.)

Let $L$ be the complex line spanned by $A_rv$. As $\langle A_rv,v\rangle$ is a nonzero element of $\ii\R$, we can, using Lemma \ref{lem:Nexist}, find a two-step nilpotent map $N$ with $N(L^\perp)=0$, $\im N \subset L$ and $N(v) = -A_rv$. Note that $\im A_i\subset V\subset L^\perp$ for all $i\geq 1$, so that $NA_i=0$ for such $i$. We define
\begin{align*}
\tilde g(\lambda):&= m_{\alpha,r,N} g(\lambda)= g(\lambda) + (\lambda-\alpha)^{-r} N \\
&= (\lambda-\alpha)^{-r} (A_r+N) + (\lambda-\alpha)^{-r+1} A_{r-1} +  \ldots + (\lambda-\alpha)^{-1}A_1 + A_0.
\end{align*}
For all $w\in \ker A_r$ we have, by \eqref{eq:orthcond2}, that $w\in (\im A_r)^\perp \subset L^\perp$, hence $Nw=0$. Moreover, $(A_r+N)(v) = 0$ by definition of $N$, so $\dim \ker (A_r+N)>\dim \ker A_r$, i.e. $\eps(\tilde{g})<\eps(g)$. This concludes the case $k=1$.\end{proof}

\begin{prop}\label{prop:pneqq} With the same notation as above, in the case $p\neq q$ the loop $g$ is a product of simple factors of the form $m_{\al,k,N}$, $k\in\N$, $\rank N\leq 2$ and $n_{\al,k,N}$, $k\in\N$, $\rank N=2$.
\end{prop}
\begin{proof}
We will prove the claim by induction on $\eps(g)$.

The same argument as in the previous proposition shows that we can reduce to the following situation: $V\supset V_1$ is maximal isotropic, and for $l>k>0$, we have $\im A_{l}\subset V $, but $\im A_{k}\subset V^\perp$, $\im A_k\not\subset V$. Since $p\neq q$, we have a decomposition
\[
\C^n=V\oplus (V\oplus sV)^\perp \oplus sV,
\]
and the inner product is definite on $W:=(V\oplus sV)^\perp$.

Since $A_0=\Id$, we may define $s\ge 1$ to be the integer such that
\begin{equation}\label{eq:imagesperp}
\im A_{k-1},\ldots,\im A_{k-s+1}\perp V, \text{ but }\im A_{k-s}\not\perp V.
\end{equation}
We claim that for all $-s+1\le i\le 0$,
\begin{equation}\label{eq:strangeAKinV}
A_{k+i}(K_{k+s+2i})\subset V.
\end{equation}
For $i=0$, this follows from \eqref{eq:orthcond}: for any $v\in K_{k+s}$,
\[
\langle A_{k}(v),A_{k}(v)\rangle=-\sum_{j=-s+1}^{-1} \langle \underbrace{A_{k-j}(v)}_{\in V},\underbrace{A_{k+j}(v)}_{\in V^\perp}\rangle-\sum_{j=1}^{s-1} \langle \underbrace{A_{k-j}(v)}_{\in V^\perp},\underbrace{A_{k+j}(v)}_{\in V}\rangle=0.
\]
The vector $A_k(v)$ is therefore isotropic and perpendicular to the maximal isotropic subspace $V$; it follows $A_k(K_{k+s})\subset V$.

If we assume that we have shown \eqref{eq:strangeAKinV} for $0\ge i\ge i_0+1$, we can show it for $i=i_0\ge -s+1$ again via \eqref{eq:orthcond}: for any $v\in K_{k+s+2i_0}$,
\begin{align*}
\langle &A_{k+i_0}(v),A_{k+i_0}(v)\rangle\\
&=-\sum_{j=-s-i_0+1}^{-1} \langle A_{k+i_0-j}(v),\underbrace{A_{k+i_0+j}(v)}_{\in V^\perp}\rangle - \sum_{j=1}^{s+i_0-1}\langle \underbrace{A_{k+i_0-j}(v)}_{\in V^\perp},A_{k+i_0+j}(v)\rangle=0,
\end{align*}
since e.g.~for $j>0$, we have
\[
A_{k+i_0+j}(v)\in A_{k+i_0+j}(K_{k+s+2i_0})\subset A_{k+i_0+j}(K_{k+s+2(i_0+j)})\subset V
\]
by assumption.

Let $l\ge -s+1$ be such that
\begin{equation}\label{eq:whichKsperp}
A_{k-s}(K_{k+l})\perp V \text{ but } A_{k-s}(K_{k+l+1})\not\perp V
\end{equation}
and choose a vector $v\in K_{k+l+1}$ with $A_{k-s}(v)\not\perp V$. It follows $A_{k+l}(v)\neq 0$. Look at the reality condition
\begin{align}
0&=\sum_{i+j=-s+l}\langle A_{k+i}(v),A_{k+j}(v)\rangle \nonumber\\
&=\langle A_{k-s}(v),A_{k+l}(v)\rangle+\ldots + \langle A_{k+l}(v),A_{k-s}(v)\rangle. \label{eq:realitydifficult}
\end{align}
Let us look at the summands of \eqref{eq:realitydifficult} each at a time. If neither of $i$ and $j$ is $-s$ and at least one, say $i$, is positive, the respective summand vanishes since then, $\im A_{k+i}\subset V$ and $\im A_{k+j}\subset V^\perp$. What happens if $i$ and $j$ are both nonpositive, and neither of them is equal to $-s$?

First of all, this is only possible if $-s+l\le 0$, i.e.~$l\le s$. We have $A_{k+i}(v)\in A_{k+i}(K_{k+l+1})$ and $A_{k+j}(v)\in A_{k+j}(K_{k+l+1})$, so we see from \eqref{eq:strangeAKinV}, that if $k+l+1\le k+s+2i$ or $k+l+1\le k+s+2j$, the respective summand vanishes. If neither of these inequalities is valid, it follows from $i+j=-s+l$ that
\[
k+l+1> k+s+2i=k+(l-i-j)+2i=k+l+i-j \Longrightarrow 1>i-j
\]
and
\[
k+l+1> k+s+2j=k+(l-i-j)+2j=k+l+j-i \Longrightarrow 1>j-i,
\]
so $i=j$. We therefore see: If $-s+l\le 0$ is odd or $-s+l>0$, \eqref{eq:realitydifficult} becomes
\begin{equation}\label{eq:realitydifficultodd}
\langle A_{k-s}(v),A_{k+l}(v)\rangle+\langle A_{k+l}(v),A_{k-s}(v)\rangle=0,
\end{equation}
and if $-s+l\le 0$ is even, \eqref{eq:realitydifficult} becomes
\begin{equation}\label{eq:realitydifficulteven}
\langle A_{k-s}(v),A_{k+l}(v)\rangle + \langle A_{k+\frac{-s+l}{2}}(v),A_{k+\frac{-s+l}{2}}(v)\rangle + \langle A_{k+l}(v),A_{k-s}(v)\rangle=0.
\end{equation}

If we are dealing with \eqref{eq:realitydifficultodd}, we have
\[
\langle A_{k-s}(v),A_{k+l}(v)\rangle\in \ii\cdot \R,
\]
so by Lemma \ref{lem:Nexist} there exists a anti-self-adjoint two-step nilpotent at most rank 2 map $N$ with $N(V^\perp)=0$ and $N(A_{k-s}(v))=-A_{k+l}(v)$. We claim that the loop
\[
\tilde{g}(\la)=m_{\al,s+l,N}(\la)g(\la)=:\sum_i (\la-\al)^{-i}\tilde{A}_i
\]
satisfies $\eps(\tilde{g})<\eps(g)$. Since $N(V^\perp)=0$, we have $\tilde{A}_i=NA_{i-s-l}+A_i=A_i$ for all $i>k+l$ by \eqref{eq:imagesperp}. For $i=k+l$,
\[
\tilde{A}_{k+l}(K_{k+l})=(NA_{k-s}+A_{k+l})(K_{k+l})=NA_{k-s}(K_{k+l})\subset N(V^\perp)=0
\]
by \eqref{eq:whichKsperp} and $\tilde{A}_{k+l}(v)=NA_{k-s}(v)+A_{k+l}(v)=0$ although $A_{k+l}(v)\neq 0$, so $\eps(\tilde{g})<\eps(g)$ and we may use induction.

If we are dealing with \eqref{eq:realitydifficulteven}, there are three subcases. If $\langle A_{k-s}(v),A_{k+l}(v)\rangle\in \ii\cdot \R$, the middle summand vanishes, so the same argument as before applies with an anti-self-adjoint map $N$ sending $A_{k-s}(v)$ to $-A_{k+l}(v)$.

If $\langle A_{k-s}(v),A_{k+l}(v)\rangle\in \R\setminus\{0\}$, we can apply Lemma \ref{lem:N2exist} to find a rank 1 map $M:(V\oplus sV)^\perp\to V$ such that $N=M-M^*$ satisfies
\begin{equation}\label{eq:newk+l}
\frac{1}{2}N^2(A_{k-s}(v))+N(A_{k+\frac{-s+l}{2}}(v))+A_{k+l}(v)=0;
\end{equation}
we want to show that the product
\begin{align*}
\tilde{g}(\la)&=n_{\al,\frac{s+l}{2},N}(\la)g(\la)=\left(\Id+(\la-\al)^{-\frac{s+l}{2}} N+\frac{1}{2}(\la-\al)^{-s-l} N^2\right)g(\la)\\
&=:\sum_i (\la-\al)^{-i} \tilde{A}_i
\end{align*}
satisfies $\eps(\tilde{g})<\eps(g)$. For that, we claim that
\[
\tilde{A}_{k+l+j}(K_{k+l+j})=\left(\frac{1}{2}N^2A_{k-s+j}+NA_{k+\frac{-s+l}{2}+j} +A_{k+l+j}\right)(K_{k+l+j})=0
\]
for all $j\ge 0$. The first summand vanishes for $j= 0$ since $A_{k-s}(K_{k+l})\subset V^\perp$ by \eqref{eq:whichKsperp}, and for $j> 0$ since then, $\im A_{k-s+j}\subset V^\perp$ by \eqref{eq:imagesperp}. The third summand vanishes trivially, so it remains to regard the second. By \eqref{eq:strangeAKinV},
\[
V\supset A_{k+\frac{-s+l}{2}+j}(K_{k+s+2\frac{-s+l}{2}+2j})=A_{k+\frac{-s+l}{2}+j}(K_{k+l+2j})\supset A_{k+\frac{-s+l}{2}+j}(K_{k+l+j}),
\]
so the second summand vanishes as well. Then, \eqref{eq:newk+l} shows $\eps(\tilde{g})<\eps(g)$.

The third case is that $\langle A_{k-s}(v),A_{k+l}(v)\rangle$ is neither real nor purely imaginary. The idea is to multiply with a simple factor of the type $m$ to make this inner product purely real. Let $N$ be a anti-self-adjoint map with $N(V^\perp)=0$ and $N^2=0$ such that
\[
N(A_{k-s}(v))=-\ii\cdot \frac{\Im \langle A_{k-s}(v),A_{k+l}(v)\rangle}{\langle A_{k-s}(v),A_{k+l}(v)\rangle}\cdot A_{k+l}(v);
\]
Lemma \ref{lem:Nexist} allows us to do so since
\begin{align*}
\langle A_{k-s}(v), -\ii\cdot \frac{\Im \langle A_{k-s}(v),A_{k+l}(v)\rangle}{\langle A_{k-s}(v),A_{k+l}(v)\rangle}\cdot A_{k+l}(v)\rangle=-\ii\cdot \Im \langle A_{k-s}(v),A_{k+l}(v)\rangle\in \ii\cdot \R.
\end{align*}
Consider the product
\begin{align*}
\tilde{g}(\la)=m_{\al,s+l,N}(\la)g(\la)=\sum_i (\la-\al)^{-i} \tilde{A}_i,
\end{align*}
where $\tilde{A}_i=NA_{i-s-l}+A_i$. Since the kernel of $N$ contains $V^\perp$, we have $\tilde{A}_i=A_i$ for all $i> k+l$, and for $i\le k+l$, they differ only by endomorphisms with values in $V$. We see that
\begin{align*}
\langle\tilde{A}_{k-s}(v),\tilde{A}_{k+l}(v)\rangle &=\langle \tilde{A}_{k-s}(v),NA_{k-s}(v)+A_{k+l}(v)\rangle \\
&=-\ii\cdot \Im \langle A_{k-s}(v),A_{k+l}(v)\rangle+\langle A_{k-s}(v),A_{k+l}(v)\rangle\in \R,
\end{align*}
and
\begin{align*}
2\cdot\langle &\tilde{A}_{k-s}(v),\tilde{A}_{k+l}(v)\rangle + \langle \tilde{A}_{k+\frac{-s+l}{2}}(v),\tilde{A}_{k+\frac{-s+l}{2}}(v)\rangle\\
&=\langle A_{k-s}(v),A_{k+l}(v)\rangle + \langle A_{k+\frac{-s+l}{2}}(v),A_{k+\frac{-s+l}{2}}(v)\rangle + \langle A_{k+l}(v),A_{k-s}(v)\rangle\\
&=0,
\end{align*}
so we have reduced to the assumptions of case two: we can now multiply $\tilde{g}$ with a simple factor of the type $n_{\al,\frac{s+l}{2},N}$ for an appropriate $N$ as explained in the previous case.
\end{proof}
By the above discussion, we use Lemma \ref{lem:Nexist} to construct $m_{\a,k,N}$, $\rank N\leq2$ to prove Propositions \ref{prop:p=q} and \ref{prop:pneqq}. Though we can only choose $N$ at most rank 2 in Lemma \ref{lem:Nexist}, here we can actually reduce the value of rank of $N$ by the factorization of $m_{\a,k,N}$, $\rank N=2$. This result can be stated as following:
\begin{prop}\label{Prop:fN2tN1}
The loops  $m_{\a,k,N}$, $\rank N=2$ can be written as $m_{\a,k,N_1}m_{\a,k,N_2}$, where $\rank N_i=1$, $N_i^2=0$, $N_i^*=-N_i$, $i=1,2$.
\end{prop}
\begin{proof}
Since $N$ satisfies $N^*=-N$ in $m_{\a,k,N}$, then $Ns$ is skew-Hermitian with respect to the standard inner product. Then $Ns=U\diag(it_1,it_2,0,\ldots,0)\bar U^t$ for some $U\in \Un_n$ and $t_i\in\R$, $i=1,2$, since $\rank N=2$, here we can assume $t_i\neq0$. Then $N=N_1+N_2$, where $N_1=U\diag(it_1,0,0,\ldots,0)\bar U^ts$, $N_2=U\diag(0,it_2,0,\ldots,0)\bar U^ts$. It is easy to see $N_i^*=-N_i$ and $N_i^2=0$, $[N_1, N_2]=0$ by $N^2=0$. Then we have $m_{\a,k,N}=m_{\a,k,N_1}m_{\a,k,N_2}$.
\end{proof}

\section{The generators for $\lupq$}\label{sec:pmnGUpq}
In this section we find generators for the group of $\GL(n,\C)$-valued rational  loops satisfying the $\Un_{p,q}$-reality condition. Besides the nilpotent loops defined in Table \ref{tab:upqn:pneqq}, we also consider the projective loops $p_{\al,L}$, $q_{\alpha,\beta,L}$ defined in Tables \ref{tab:upq:p=q} and \ref{tab:upq:pneqq} below. We will prove that these nilpotent and projective loops together generate  $\lupq$.

The complex line $L$ in $p_{\al,L}$ and $q_{\alpha,\beta,L}$ can be replaced by vector subspace $V$, i.e. we can also define $p_{\al,V}$ and $q_{\alpha,\beta,V}$ similarly. One easily sees that the first two types of simple elements satisfy the reality condition. Also, the loops $p_{\al,L}$ already appear in \cite{Ter00}, Section 11. Note that the $q_{\al,\beta,L}$ are products of two $\GL(n,\C)$-simple elements: $q_{\al,\beta,L}=p_{\al,\beta,L,(sL)^\perp}p_{\bar{\beta},\bar{\al},sL,L^\perp}$. Furthermore, there is an overlap between the first two types: $ q_{\al,\ol{\al},L}=p_{\al,L\oplus sL}$. Since $\C^n=L\oplus L^\perp,\, L\cap L^\perp=0$ in $p_{\al,L}$, then $\pi_L$ in $p_{\al,L}$ satisfies $\pi_L^{*}=\pi_L$, and $L$ is isotropic complex line in $q_{\alpha,\beta,L}$, i.e. $\C^n=L\oplus (L\oplus sL)^\perp \oplus sL$ is Hermitian orthogonal decomposition, then $\pi_L$ in $q_{\alpha,\beta,L}$ satisfies $\bar\pi_L^{t}=\pi_L$.
\begin{table}[h]
\caption{Generators for $\lupp$}
\begin{tabular}{|c|c|c|}
\hline
Name & Definition & Conditions\\
\hline
$p_{\al,L}$ & $\left(\frac{\la-\al}{\la-\bar{\al}}\right)\pi_L+\pi_{L^\perp}$ & \begin{tabular}{c} $\C^n=L\oplus L^\perp,\, L\cap L^\perp=0$,\\ $L$ complex line\end{tabular}\\
\hline
$q_{\alpha,\beta,L}$ & $\left(\frac{\la-\al}{\la-\beta}\right)\pi_L+\pi_{(L\oplus sL)^\perp} +\left(\frac{\la-\ol{\beta}}{\la-\ol{\al}}\right) \pi_{sL}
$ & \begin{tabular}{c} $\C^n=L\oplus (L\oplus sL)^\perp \oplus sL$ \\ $L$ isotropic complex line\end{tabular}\\
\hline
$m_{\al,k,N}$ & $\Id+\left(\frac{1}{\la-\al}\right)^k N$ & \begin{tabular}{c} $\al\in\R$,  $k=1, 2$, \\ $N^2=0$,
 $N^*=-N$, $\rank N=1$ \end{tabular}\\
\hline
\end{tabular}
\label{tab:upq:p=q}
\end{table}
\begin{table}[h]
\caption{Generators for $\lupq, (p\neq q)$}
\begin{tabular}{|c|c|c|}
\hline
Name & Definition & Conditions\\
\hline
$p_{\al,L}$ & $\left(\frac{\la-\al}{\la-\bar{\al}}\right)\pi_L+\pi_{L^\perp}$ & \begin{tabular}{c} $\C^n=L\oplus L^\perp,\, L\cap L^\perp=0$,\\ $L$ complex line\end{tabular}\\
\hline
$q_{\alpha,\beta,L}$ & $\left(\frac{\la-\al}{\la-\beta}\right)\pi_L+\pi_{(L\oplus sL)^\perp} +\left(\frac{\la-\ol{\beta}}{\la-\ol{\al}}\right) \pi_{sL}
$ & \begin{tabular}{c} $\C^n=L\oplus (L\oplus sL)^\perp \oplus sL$ \\ $L$ isotropic complex line\end{tabular}\\
\hline
$m_{\al,k,N}$ & $\Id+\left(\frac{1}{\la-\al}\right)^k N$ & \begin{tabular}{c} $\al\in\R$,  $k=1, 2$, \\ $N^2=0$,
 $N^*=-N$, $\rank N=1$ \end{tabular}\\
\hline
$n_{\al,k,N}$ & $\Id+\left(\frac{1}{\la-\al}\right)^k N+\frac{1}{2}\left(\frac{1}{\la-\al}\right)^{2k} N^2$ & \begin{tabular}{c} $\al\in \R$, $k=1$, $N=M-M^*$, where\\ $M:(V\oplus sV)^\perp\to V$,\\ $V$ max.~isotropic,\\$\rank M=1$, $\rank N=2$ \end{tabular}\\
\hline
\end{tabular}
\label{tab:upq:pneqq}
\end{table}

Compared to the nilpotent loops in Table \ref{tab:upqn:pneqq} we have restricted the values of the number $k$. This is possible because of the following proposition:
\begin{prop}\label{Ureducek}
The loops  $m_{\a,k,N}$, $k\in \N$ (resp. $n_{\a,k, N}$, $k\in\N$) can be factored as the products of some $q_{\a,\beta, L}$ and  $m_{\a,k,rN}$, $k=1,2$, $r\in\R$ (resp. $n_{\a,1, rN}$, $r\in\R$).
\end{prop}
\begin{proof}
For the factorization of $m_{\a,k,N}$, we just prove the case $k>2$. Here we can assume $\rank N=1$.  Assume $L=\im N$, then by Lemma \ref{lem:Nexist}, we have $L$ is isotropic and $\ker N=L^\bot$. Then choose $\beta\in \R$, $\beta\neq\a$, we have
\begin{eqnarray*}
&&q_{\a,\beta, L} m_{\a,k,N} q_{\a,\beta, L}^{-1}
=\exp[\Ad (q_{\a,\beta, L}) \frac{1}{(\l-\a)^k}N]\\
&&=\exp[\frac{1}{(\l-\a)^{k-2}(\l-\beta)^2 }\pi_{L}N\pi_{sL}]
=\exp[\frac{1}{(\l-\a)^{k-2}(\l-\beta)^2 }N].
\end{eqnarray*}
 Then we can use partial fraction decomposition to decompose the above formula as the product of factors of the form $m_{\a,l_{\a},r_1 N}$, $1\leq l_{\a}\leq k-2$, $r_1\in\R$,  $m_{\beta,l_{\beta},r_2 N}$, $1\leq l_{\beta}\leq 2$, $r_2\in\R$ and apply the induction hypothesis, to  $m_{\a,l_{\a},r_1 N}$.\\
\indent For the factorization of $n_{\a,k,N}$, we just prove the case $k>1$. Notice $N=M-M^{*}$, and $M: (V\oplus sV)^\perp\longrightarrow V$, $M(V\oplus sV)=0$, $V^\perp \subset \ker M^*$, $M^*(sV)\subset (V\oplus sV)^\perp$ for some maximal isotropic subspace $V$. Then choose $\beta\in \R$, $\beta\neq\a$, we have
\begin{eqnarray*}
&&q_{\a,\beta, V} n_{\a,k,N} q_{\a,\beta, V}^{-1}=\exp[\Ad (q_{\a,\beta, V}) \frac{1}{(\l-\a)^k}(M-M^*)]\\
&&=\exp[\frac{1}{(\l-\a)^{k-1}(\l-\beta) }(\pi_{V}M\pi_{(V\oplus sV)^\bot}-\pi_{(V\oplus sV)^\bot}M^*\pi_{sV})]\\
&&=\exp[\frac{1}{(\l-\a)^{k-1}(\l-\beta) } (M-M^*)],
\end{eqnarray*}
 where we use $\pi_{sV}^*=\pi_{V}$ , $(\pi_{(V\oplus sV)^\bot})^*=\pi_{(V\oplus sV)^\bot}$. Notice $N=M-M^*$, then we can use partial fraction decomposition to decompose the above formula as the product of factors of the form $n_{\a,l_{\a},r_1 N}$, $1\leq l_{\a}\leq k-1$, $r_1\in\R$, $n_{\beta,1,r_2 N}$, $r_2\in\R$ and apply the induction hypothesis, to  $n_{\a,l_{\a},r_1 N}$.\\
\indent Finally we just notice any $q_{\a,\beta,V}$ can be factored as the product of some $q_{\a,\beta, L}$, then we can get the final proof of this proposition.
\end{proof}

Now we can give the proof of the generating theorem of rational loop group $\lupq$. By the discussing above, we just prove any $\lupq$ can be factored as the product of some generators and a negative algebraic loop with only one real singularity.
\begin{thm}\label{Upq:ge}
The rational loop group  $\lupq$ is generated by the simple elements given in Table \ref{tab:upq:p=q} if $p=q$ and Table \ref{tab:upq:pneqq} if $p\neq q$ .
\end{thm}
\begin{proof}
Let $g\in\lupq$. Observe that $g$ is holomorphic (in particular $\GL(n,\C)$-valued) at a point $\al$ if and only if it is holomorphic at $\ol{\al}$. We proceed in three steps as in the proof of \cite[Theorem 4.1]{Goe13}.

Let us first consider the case of $g$ having more than one singularity, and not all of them real. Let $\al\in \C\setminus \R$ be one of them, and choose $\beta\neq \al,\bar{\al}$ to be another (real or complex) singularity; if there is none, let $\beta$ be an arbitrary real number. If $\al$ is a pole, write down the Laurent expansion of $g$ in $\frac{\la-\al}{\la-\beta}$ around $\al$ as
\[
g(\la)=\sum_{j=-k}^\infty \left(\frac{\la-\al}{\la-\beta}\right)^j g_j
\]
with $g_{-k}\neq 0$; otherwise continue with \eqref{eq:nopoleupq}. If there exists an isotropic complex line  $L\subset \im g_{-k}$, then the loop
\begin{align*}
q_{\al,\beta,L}(\la) g(\la)&=\left(\left(\frac{\la-\al}{\la-\beta}\right)\pi_L+\pi_{(L\oplus sL)^\perp} + \left(\frac{\la-\ol{\beta}}{\la-\ol{\al}} \right)\pi_{sL}\right)g(\la)\\
&=\left(\frac{\la-\beta}{\la-\al}\right)^k \left(\pi_{(L\oplus sL)^\perp}+\left(\frac{\al-\bar{\beta}}{\al-\bar{\al}}\right)\pi_{sL}\right)g_{-k}+\ldots
\end{align*}
has a pole of lower degree at $\al$; note that we used here that $\al$ is nonreal.

If there exist no isotropic complex lines of $\im g_{-k}$, let $L\subset\im g_{-k}$ and note that $L\cap L^\perp=0$. Then,
\begin{align*}
p_{\al,L}(\la)g(\la)&=\left(\left(\frac{\la-\al}{\la-\bar{\al}}\right)\pi_L+\pi_{L^\perp}\right)\left(\left(\frac{\la-\beta}{\la-\al}\right)^kg_{-k}+\ldots\right)\\
&=\left(\frac{\la-\beta}{\la-\al}\right)^k\pi_{L^\perp}\circ g_{-k}+\ldots,
\end{align*}
has a pole of lower degree at $\al$.

Repeating this, we obtain a loop $g$ without pole at $\al$, whose Laurent expansion we write as
\begin{equation}\label{eq:nopoleupq}
g(\la)=g_0+\left(\frac{\la-\al}{\la-\beta}\right)g_1+\ldots.
\end{equation}
If $g_0$ is invertible, we have removed the singularity at $\alpha$, so assume that $g_0$ is singular. As in the previous proofs, we try to reduce the order of the zero of the map $\la\mapsto \det g(\la)$ by multiplying simple factors.

If $\im g_0\cap (\im g_0)^\perp=0$, it is easy to see $\C^n=\im g_0 \oplus (\im g_0)^\perp$. Since $g_0$ is singular, then dim $(\im g_0)^\perp >=1$. We will prove there exists a complex line $L\subset (\im g_0)^\perp$ such that $\langle L, L\rangle\neq 0$. If dim $(\im g_0)^\perp =1$, since $((\im g_0)^\perp)^\perp=\im g_0$, we have got the conclusion. Next we consider the case dim $(\im g_0)^\perp>1$. Assume for any complex line $L\subset (\im g_0)^\perp$, we have $\langle L, L\rangle= 0$, then we can choose $u, v\in(\im g_0)^\perp$ such that $\langle u, v\rangle \neq 0$, if not, we must have $(\im g_0)^\perp$ is isotropic, then $\im g_0=((\im g_0)^\perp)^\perp\supset (\im g_0)^\perp$, it is a contradiction. Then we have
\begin{eqnarray*}
0=\langle \l u+v, \l u+v\rangle =|\l|^2\langle u, u\rangle+ 2\Re \bar\l \langle u,v\rangle +\langle v, v\rangle=2\Re \bar\l \langle u,v\rangle
 \end{eqnarray*}
for all $\l\in \C$. It is a contradiction. Assume $L\subset (\im g_0)^\perp$(i.e. $\im g_0\subset L^\perp$) such that $\langle L, L\rangle\neq 0$, then we can reduce the order of the zero by regarding $p_{\bar{\al},L}g$. \\
\indent If $\im g_0\cap (\im g_0)^\perp\neq 0$, let $V:=\im g_0\cap (\im g_0)^\perp$ and complex line $L\subset V$. Note that $L$ is isotropic and $\im g_0\subset L^\perp$. Then the order of the zero is now reduced by regarding $q_{\beta,\al,sL}g$

Repeating this, we obtain a loop that  has no singularity at $\al$.

\vspace{0.4cm} If $g$ has several singularities, but all of them are real, let $\al$ and $\beta$ be two of those. Because $\al$ and $\beta$ are real, the reality condition says that they are automatically poles (if e.g.~$\al$ was not a pole, then $g(\al)^*g(\al)=\Id$ and hence $g(\al)$ would be invertible). Expanding $g$ in a Laurent series as before, it follows from the reality condition that $V:=\im g_{-k}$ is isotropic and that
\begin{equation}\label{eq:upqrealitycond2nd}
\langle g_{-k+1}v,g_{-k}w\rangle + \langle g_{-k} v,g_{-k+1}w\rangle=0
\end{equation}
for all $v$ and $w$. Let $L\subset V$ be a line, and let $v$ be such that $g_{-k}(v)$ spans $L$. There are two cases: either
\begin{equation}\label{eq:upqrealitycond3rd}
\langle g_{-k+1}v,g_{-k}v\rangle=0
\end{equation}
or
\begin{equation}\label{eq:upqrealitycond4th}
\langle g_{-k+1}v,g_{-k}v\rangle\in \ii\cdot \R\setminus\{0\}.
\end{equation}
In the first case we claim that the product $q_{\beta,\al,sL}g$ has a pole of lower degree than $g$ at $\al$. To see this, we calculate
\[
q_{\beta,\al,sL}(\la)g(\la)=\left(\frac{\la-\beta}{\la-\al}\right)^k(\pi_{(L\oplus sL)^\perp} g_{-k} + \pi_{sL}g_{-k+1})+\ldots
\]
Denoting $\tilde{g}_{-k}:=\pi_{(L\oplus sL)^\perp} g_{-k} + \pi_{sL}g_{-k+1}$, we need to show that $\dim \ker g_{-k}<\dim \ker \tilde{g}_{-k}$. Equation \eqref{eq:upqrealitycond2nd} implies that $g_{-k+1}(\ker g_{-k}) \perp \im g_{-k}$, hence $\ker g_{-k}\subset \tilde{g}_{-k}$. Furthermore if $v$ is such that $g_{-k}(v)$ spans $L$, then \eqref{eq:upqrealitycond3rd} implies  that $v\in \ker \tilde{g}_{-k}$.

In the second case, i.e.\ if \eqref{eq:upqrealitycond4th} holds, we have $g_{-k+1}(v)\notin L^\perp$. Applying Lemma \ref{lem:Nexist} we find a two-step nilpotent anti-self-adjoint endomorphism $N$ with $N(L^\perp)=0$, $\im N \subset L$ and $N(g_{-k+1}(v)) = (\beta-\alpha)g_{-k}(v)$. We claim that the product $m_{\alpha,1,N}g$ has a pole of lower degree than $g$ at $\alpha$. We compute using $N g_{-k} = 0$:
\begin{align*}
m_{\alpha,1,N}(\lambda)g(\lambda) &= \left(\frac{\lambda-\beta}{\lambda-\alpha}\right)^k g_{-k} + \frac{1}{\lambda-\alpha} \left(\frac{\lambda-\beta}{\lambda-\alpha}\right)^{k-1} N g_{-k+1} + \ldots\\
&= \left(\frac{\lambda-\beta}{\lambda-\alpha}\right)^k \left(g_{-k} + \frac{1}{\lambda-\beta} Ng_{-k+1}\right) + \ldots \\
&=\left(\frac{\lambda-\beta}{\lambda-\alpha}\right)^k \left(g_{-k} + \frac{1}{\alpha-\beta} Ng_{-k+1}\right) + \ldots
\end{align*}
Now, denoting $\tilde{g}_{-k} = g_{-k} + \frac{1}{\alpha-\beta} Ng_{-k+1}$, we have $\ker g_{-k}\subset \ker \tilde{g}_{-k}$ as well as $\tilde{g}_{-k}(v)=0$, so the order of the pole at $\alpha$ has decreased.

Repeating this, we obtain a loop $g$ without  singularity at $\al$. Then we can assume $g$ is a negative algebraic loop with only one real singularity. Then Theorem \ref{Upq:facnalstonil} and Proposition \ref{Ureducek} complete the proof.
\end{proof}

\section{Projective generators for $\lupq$}\label{sec:pGUpq}

Note that the generators $q_{\alpha,\beta,L}$ are defined via the decomposition $\C^n=L\oplus(L\oplus sL)^\perp\oplus sL$, where $L$ is an isotropic line (or subspace). This data is clearly not preserved by the adjoint action of $\Un_{p,q}$. In the following proposition we generalize the $q_{\alpha,\beta,L}$ to an $\Ad\, \Un_{p,q}$-invariant set of rational loops.

 \begin{prop}\label{prop:proj2}
Choose any two isotropic complex lines $L, K$ in general position, i.e., $\langle L,K\rangle\neq0$. There is a natural vector space decomposition: $\C^n = L \oplus K \oplus(L^\perp \cap K^\perp)$ inducing $3$ projections onto each component along the rest. Let $W$ denote the third component of dimension $(n-2)$. Then $L \oplus W = L^\perp$, $K \oplus W = K^\perp$, $L \oplus K=W^\perp$, and these three projections are naturally denoted by $\pi_{L,K^\perp}$, $\pi_{K,L^\perp}$, $\pi_{W,W^\perp}$. The following rational loop lies in $\lupq$ and is equal to $q_{\alpha,\beta,L}$ when $K=sL$:
\[
g_{\a,\beta, L, K, W}:=\left(\frac{\l-\a}{\l-\beta}\right)\pi_{L,K^\perp}+\left(\frac{\l-\bar\beta}{\l-\bar\a}\right)\pi_{K,L^\perp} + \pi_{W,W^\perp} .
\]
Conversely, for any rank $1$ projection $\pi$ satisfying $ \pi \pi^* =  \pi^* \pi =0 $ and for which $\Im(\pi)$ and $\Im(\pi^*)$ are two isotropic complex lines in general position, we recover the above construction via
\[
L:=\Im(\pi), \quad
K:=\Im(\pi^*), \quad
\pi_{L,K^\perp}=\pi, \quad
\pi_{K,L^\perp}=\pi^*, \quad
\pi_{W,W^\perp}= I - \pi - \pi^*.
\]

\end{prop}

\begin{proof}
Direct computations based on the formula $\pi_{L,K^\perp} = \frac{L\bar{K}^t s}{\bar{K}^t sL} $, where we denote by $L$ and $K$ also the basis column vector of the respective line. It is easy to check that $\pi_{L,K^\perp}^*$ projects onto $K$ along $L^\perp$, i.e., $\pi_{L,K^\perp}^*=\pi_{K,L^\perp}$. It follows from $\pi_{L,K^\perp}\pi_{L,K^\perp}^*=\pi_{L,K^\perp}^* \pi_{L,K^\perp}=0$ that $I-\pi_{L,K^\perp}-\pi_{L,K^\perp}^*$ is also a projection. In particular, when $K=sL$, the above loop $g_{\a,\beta, L, K, W}$ is identical to  $q_{\alpha,\beta,L}$ in Table \ref{tab:upq:p=q} or \ref{tab:upq:pneqq}. 
\end{proof}
The conditions for the above construction are all naturally $\Ad\, \Un_{p,q}$-invariant. We have noted recently that a sim lar construction was given by Terng-Wu \cite{Ter16} in the case of $O_{p,q}$, and we will present the corresponding projective generators in a subsequent work.

When no confusion arises, we denote the three projections simply by $\pi_{L},\pi_{K}$ and $\pi_{W}$. We will also use the notation
\[
g_{\a,\beta, \pi}:=\frac{\l-\a}{\l-\beta} \, \pi+ \frac{\l-\bar\beta}{\l-\bar\a} \, \pi^* + (I - \pi - \pi^*) = I + \frac{\beta-\alpha}{\lambda-\beta} \, \pi + \frac{\bar\a-\bar\beta}{\l-\bar\a} \, \pi^*.
\]

It will turn out that the projective loops $g_{\a,\beta,  L, K, W}$ and $p_{\al,L}$ together generate the whole group $\lupq$, similar to $\lglc$ and $\lglr$, see Theorem \ref{thm:fac4} below. We summarize them in a table:

\begin{table}[h]
\caption{Projective generators for $\lupq$}
\begin{tabular}{|c|c|c|}
\hline
Name & Definition & Conditions\\
\hline
$p_{\al,L}$ & $\left(\frac{\la-\al}{\la-\bar{\al}}\right)\pi_L+\pi_{L^\perp}$ & \begin{tabular}{c} $\C^n=L\oplus L^\perp,\, L\cap L^\perp=0$,\\ $L$ non-isotropic complex line\end{tabular}\\
\hline
$g_{\alpha,\beta,  L, K, W}$ & \begin{tabular}{c} $ \left(\frac{\l-\a}{\l-\beta}\right)\pi_{L}+\left(\frac{\l-\bar\beta}{\l-\bar\a}\right)\pi_{K} + \pi_{W} $ \end{tabular} & \begin{tabular}{c}   
$\C^n = L \oplus K \oplus W$, \; $W:=L^\perp \cap K^\perp$, \\ $L,K$ isotropic complex lines,  $\langle L, K\rangle\neq0$ \end{tabular}\\
\hline
\end{tabular}
\label{tab:upq2}
\end{table}

To prove this generating theorem, it is, by Theorem \ref{Upq:ge}, enough to show that the nilpotent loops $m_{\a,k,N}$, $n_{\a,k,N}$ in Tables \ref{tab:upq:p=q} and \ref{tab:upq:pneqq} can be factored as the product of some $g_{\alpha,\beta,L,K,W}$. We need a matrix characterization of $N$ in $m_{\a,k,N}$.

\begin{lem} \label{nil}
Assume that $N^{*}=-N$, $N^2=0$, and that $\rank(N)=1$. Let $L$ denote $\im N$ (as well as its basis column vector). Then $N=\ii r L \bar L^t s$ for some $r \in\R\setminus\{0\}$.
\end{lem}

\begin{proof}
The map $N$ is, up to a real scalar, uniquely determined by the conditions $\im N = L$ and $N^*=-N$. The claim follows because the linear map $\ii L\bar L^t s$ satisfies the same conditions.
\end{proof}

To show the difference between the  general projection $\pi_{L, K^\perp}$ and the Hermitian orthogonal projection $\pi_{L, sL^\perp}$, we choose $K$ to be a simple geometric perturbation of $L$ and obtain the following interesting formulas.

\begin{lem} \label{difference}
Choose a basis of any isotropic complex line $L:=\bpm L_1\\ L_2\epm$ (column vector), satisfying that $L_1$ and $L_2$ are $p$- and $q$-dimensional unit vectors (standard Hermitian norm), respectively. Set $K:=\bpm L_1\\ e^{\ii \theta }L_2\epm$ for $\theta\in \R\setminus \{k\pi\}_{k\in \Z}$. Then $L$ and $K$ are two isotropic lines in general position: $\langle L,K\rangle\neq 0$, and
\[
 \pi_{L, K^\perp} - \pi_{L, sL^\perp} =  \pi_{L,K^\perp}  \cdot  \pi_{sL,L^\perp} = \frac{(1+e^{\ii \theta})}{2(1-e^{\ii \theta})}L\bar {L}^t s  \;  ,
\]
which are all anti-self-adjoint and $2$-step nilpotent.
\end{lem}

\begin{proof}
Direct computations based on the formula $\pi_{L,K^\perp} = \frac{L\bar{K}^t s}{\bar{K}^t sL} $\,:
\begin{eqnarray*}
\pi_{L, K^\perp}-\pi_{L, sL^\perp} &=& \frac{L\bar K^t s}{\bar K^t s L} - \frac{L\bar L^t}{\bar L^t L}\\
            &=&\frac{\bpm L_1\\L_2\epm \bpm -\bar{L_1}^t & e^{-\ii\theta}\bar{ L_2}^t\epm}{-1+e^{-\ii\theta}}-\frac{\bpm L_1 \\ L_2\epm\bpm \bar{L_1}^t & \bar{L_2}^t\epm}{2}\\
            &=&\frac{1}{2(-1+e^{-\ii\theta})}\bpm (-1-e^{-\ii\theta})L_1\bar{L_1}^t & (1+e^{-\ii\theta}) L_1\bar{ L_2}^t \\(-1-e^{-\ii\theta})L_2\bar{L_1}^t &(1+e^{-\ii\theta})L_2\bar{ L_2}^t\epm\\
            &=&\frac{(1+e^{\ii \theta})}{2(1-e^{\ii \theta})} L\bar {L}^t s.
\end{eqnarray*}

\begin{eqnarray*}
\pi_{L,K^\perp}\pi_{sL,L^\perp} &=& \frac{L\bar K^t s}{\bar K^t s L} \cdot \frac{s L\bar L^t s}{\bar L^t s s L}\\
            &=&\frac{\bpm L_1\\L_2\epm \bpm -\bar{L_1}^t & e^{-\ii\theta}\bar{ L_2}^t\epm}{-1+e^{-\ii\theta}} \cdot \frac{\bpm -L_1 \\ L_2\epm\bpm - \bar{L_1}^t & \bar{L_2}^t\epm}{2}\\
            &=&\frac{(1+e^{-\ii \theta})}{2(-1+e^{-\ii\theta})}\bpm -L_1\bar{L_1}^t &  L_1\bar{ L_2}^t \\- L_2\bar{L_1}^t & L_2\bar{ L_2}^t\epm\\
            &=&\frac{(1+e^{\ii \theta})}{2(1-e^{\ii \theta})} L\bar {L}^t s.
\end{eqnarray*}
We observe that $N$ in Lemma \ref{nil} has exactly the same form as the above difference or product, since the M\"{o}bius transformation $\frac{1+z}{1-z}$ maps $S^{1}\setminus\{1,-1\}$ onto $\ii\R\setminus\{0\}$. So they are all anti-self-adjoint and $2$-step nilpotent.

\end{proof}

To make the above formulas more clear, let $\pi:=\pi_{L,K^\perp}$ and $\tilde{\pi}:=\pi_{L,sL^\perp}$. The following basic formulas are immediate from the definition of projections:
\[
\pi \cdot \tilde{\pi} = \tilde{\pi} , \quad
 \tilde{\pi} \cdot \pi = \pi , \quad
\pi^\ast \cdot \tilde{\pi} = 0 \; .
\]
The above Lemma \ref{difference}  however gives less trivial formulas:
\[
\pi - \tilde{\pi} \, = \, \pi \cdot \tilde{\pi}^\ast \, = \, N \; ,
\]
where \, $N:=\frac{(1+e^{\ii \theta})}{2(1-e^{\ii \theta})} L\bar {L}^t s $ \, is anti-self-adjoint and $2$-step nilpotent. One geometric meaning of the first equality is that $ \Im (I-\tilde{\pi}-\tilde{\pi}^*)$ or $L^\perp \cap sL^\perp$ is inside $\ker \pi$, and is actually equal to $W$.

In fact, one hint for our factorization of nilpotent loops was this factorization of $N$ as the product of some projections. Hawkins-Kammerer proved in \cite{haw68} that any singular finite-dimensional linear operator can be factored as the product of some projections.

Recall that the original $q_{\alpha,\beta,L}$ is identical to $g_{\a,\beta, L,sL, L^\perp \cap sL^\perp }$, and can be written using $\tilde{\pi}=\pi_{L,sL^\perp}$ as:
\[
q_{\alpha,\beta,L} = g_{ \a,\beta, \tilde{\pi} } = I + \frac{\beta-\alpha}{\lambda-\beta} \, \tilde{\pi} + \frac{\bar\a-\bar\beta}{\l-\bar\a} \, \tilde{\pi}^* \; .
\]

\begin{prop}\label{prop:facm}
 The nilpotent loop $m_{\a,k,N}$, $k=1,2$ defined in Tables \ref{tab:upq:p=q} and \ref{tab:upq:pneqq} can be factored as the product of some $g_{\a,\beta, L, K, W}$.
\end{prop}

\begin{proof}
For a nilpotent loop $m_{\a,k,N}$ as defined in Tables \ref{tab:upq:p=q} and \ref{tab:upq:pneqq}, we apply Lemma \ref{difference} to the isotropic complex line $\im N$. That is, we choose a basis $L:=\bpm L_1\\ L_2\epm$ of $\im N$, where $L_1$ and $L_2$ are $p$- and $q$-dimensional unit vectors, and, for $\theta\in \R\setminus \{k\pi\}_{k\in \Z}$, define another isotropic line by $K:=\bpm L_1\\ e^{\ii \theta }L_2\epm$. Then, $\langle L,K\rangle\neq 0$ and
\[
\pi - \tilde{\pi} \, = \, \pi \cdot \tilde{\pi}^\ast \, = \, N,
\]
where $N=\frac{(1+e^{\ii \theta})}{2(1-e^{\ii \theta})} L\bar {L}^t s $ is anti-self-adjoint and $2$-step nilpotent. Recall that $m_{\a,k,N}$ has a unique real singularity $\a\in\R$. For any $\beta\in\C$ we compute using the above formulas together with the basic formulas \,
$\pi \cdot \tilde{\pi} = \tilde{\pi} ,  \;
 \tilde{\pi} \cdot \pi = \pi ,  \;
\pi^\ast \cdot \tilde{\pi} = 0$:

\begin{eqnarray*}
&    &  g_{\beta,\a,L,K,W} \cdot g_{\a,\beta, L, sL, L^\perp \cap sL^\perp } \\
& = &  ( I + \frac{\alpha-\beta}{\lambda-\alpha} \, \pi + \frac{\bar\beta-\alpha}{\l-\bar\beta} \, \pi^*) \cdot ( I + \frac{\beta-\alpha}{\lambda-\beta} \, \tilde{\pi} + \frac{\a-\bar\beta}{\l-\a} \, \tilde{\pi}^*) \\
& = &  I + \left[ \frac{\beta-\alpha}{\lambda-\beta}- \frac{(\beta-\alpha)^2}{(\lambda-\alpha)(\lambda-\beta)} \right] \tilde{\pi} + \frac{\alpha-\beta}{\lambda-\alpha}  \pi  \\
&    & + \left[ \frac{\bar\beta-\alpha}{\lambda-\bar\beta}- \frac{(\bar\beta-\alpha)^2}{(\lambda-\alpha)(\lambda-\bar\beta)} \right] \pi^*   + \frac{\alpha-\bar\beta}{\lambda-\alpha}  \tilde{\pi}^* + \frac{|\alpha-\beta|^2}{(\lambda-\alpha)^2} \pi  \tilde{\pi}^\ast \\
&  =  & I + \frac{\alpha-\beta}{\lambda-\alpha}  (\pi - \tilde{\pi} ) + \frac{\alpha-\bar\beta}{\lambda-\alpha}  (\tilde{\pi}^* - \pi^* ) + \frac{|\alpha-\beta|^2}{(\lambda-\alpha)^2} \pi  \tilde{\pi}^\ast \\
&  = & I+\frac{2\alpha-\beta-\bar \beta}{\lambda-\alpha}N+\frac{|\alpha-\beta|^2}{(\lambda-\alpha)^2} N\\
&  =  & (I + \frac{2\alpha-\beta-\bar \beta}{\lambda-\alpha} N) (I + \frac{|\alpha-\beta|^2}{(\lambda-\alpha)^2} N).
\end{eqnarray*}

Choosing  $\beta= \a + \ii$, we have $2\alpha-\beta-\bar\beta = 0$, hence it follows that $m_{\a,2,N}$ can be factored as the product of two loops of type $g$.

Choosing $\beta$ such that $2\a-\beta-\bar\beta=1$, we then see that $m_{\a,1,N}$ can be factored as the product of two projective loops and one loop of type $m$, or four projective loops. 
\end{proof}

For the factorization of $n_{\a,1,N}$ in Table \ref{tab:upq:pneqq}, we need another computational lemma. \\

\begin{lem}\label{fornupn2}
Assume that $p\neq q$,  and let $V$ be a maximal isotropic subspace with respect to $\langle \cdot, \cdot \rangle$ of signature $(p,q)$. For nonzero vectors $v \in V$ and $w \in (V\oplus sV)^\perp$, define $M$ by $M(w)=-2v$, $M(w^\perp)=0$ and define $N:=M-M^*$.  Let $L$ denote the line spanned by $v$. Let $K$ denote the line spanned by $v+cw- \mathrm{sgn}(\langle w, w\rangle) s v $  where $c=\sqrt{\frac{2\, \mathrm{sgn}(\langle w, w\rangle)\langle s v,v \rangle}{\langle w, w\rangle}} > 0$. Then
\begin{enumerate}
  \item $L,K$ are two isotropic lines in general position: $\langle L,K\rangle\neq 0$; 
  \item $\pi_{L,K^\perp}\pi_{sL,L^\perp}$ is self-adjoint; 
  \item $M=(\a-\beta)\pi_{L, K^\perp}\pi_{L^\perp\cap sL^\perp, L\oplus sL}$, where $\beta= \a - \frac{2\,\mathrm{sgn}(\langle w, w\rangle) \cdot \langle sv, v\rangle}{c \langle w,w\rangle}$; 
  \item $N^2=2|\a-\beta|^2 \pi_{L,K^\perp}  \cdot  \pi_{sL,L^\perp} $. 
\end{enumerate}
\end{lem}

\begin{proof}
Firstly $\langle w,w\rangle \neq 0$. Otherwise $V\oplus \C w$ would be a bigger isotropic subspace, contradicting to the condition that $V$ is maximal isotropic. Recall that $\mathrm{sgn}(\langle w, w\rangle)=1$, if $\langle w, w\rangle>0$, and $\mathrm{sgn}(\langle w, w\rangle)=-1$, if $\langle w, w\rangle<0$; and $\langle sv, v\rangle > 0$. For convenience, denote $v$ also by $L$, and denote the basis column vector of $K$, i.e. $v+cw - \mathrm{sgn}(\langle w, w\rangle) s v$, also by $K$. Then 
\[
\langle L, K \rangle = - \, \mathrm{sgn}(\langle w, w\rangle) \langle sv, v\rangle \in \R \setminus \{0\} 
\]
and 
\[
\langle K, K \rangle =c^2 \langle w, w\rangle - 2 \, \mathrm{sgn}(\langle w, w\rangle) \langle sv, v\rangle = 0 
\]
imply the first claim. 

For the second claim, we compute $\pi_{L,K^\perp}\pi_{sL,L^\perp}=\frac{\bar K^t L}{\bar K^t sL}\frac{L\bar L^t s}{\bar L^t L}$, which is then self-adjoint if and only if $\frac{\bar K^t L}{\bar K^t sL}\in\R$. But $\frac{\bar K^t L}{\bar K^t sL}=-\, \mathrm{sgn}(\langle w, w\rangle)$, and the claim follows.

For the third claim, we first compute
\[
(\a-\beta)\pi_{L, K^\perp}\pi_{L^\perp\cap sL^\perp, L\oplus sL} ~ w=(\a-\beta)\pi_{L, K^\perp}w=\frac{c (\a-\beta) \langle w,w\rangle}{-\, \mathrm{sgn}(\langle w, w\rangle) \langle s v, v\rangle}v = -2v.
\]
For any $u\in w^\perp$, decompose $u=x v+y s v+ z u_1$ by $\C^n=L\oplus sL\oplus (L\oplus sL)^\perp$, where $u_1\in (L\oplus sL)^\perp$, $x,y,z\in\C$. $\langle u, w\rangle=0$ implies $\langle u_1,w\rangle=0$. Then the third claim follows by the following computation:  
\begin{eqnarray*}
&&(\a-\beta)\pi_{L, K^\perp}\pi_{L^\perp\cap sL^\perp, L\oplus sL} ~ u=(\a-\beta)\pi_{L, K^\perp}u_1\\
&&=(\a-\beta)\frac{\langle v,u_1\rangle+\langle c w,u_1\rangle+\langle -\, \mathrm{sgn}(\langle w, w\rangle)s v, u_1\rangle}{\langle -\, \mathrm{sgn}(\langle w, w\rangle)s v, v\rangle}v=0. 
\end{eqnarray*}

For the last claim, we first show  $M^2=(M^\ast)^2=M^\ast M=0$. For convenience, let $\pi :=\pi_{L,K^\perp}$ and $\tilde \pi :=\pi_{L,sL^\perp}$. Then $M=(\a-\beta)\pi(I-\tilde\pi-\tilde\pi^{\ast})$, and 
\begin{eqnarray*}
M^2=(\a-\beta)^2\pi(I-\tilde\pi-\tilde\pi^{\ast})\pi(I-\tilde\pi-\tilde\pi^{\ast})=0, 
\end{eqnarray*}
by $(I-\tilde\pi-\tilde\pi^{\ast})\pi=0$; and
\begin{eqnarray*}
M^*M=|\a-\beta|^2(I-\tilde\pi-\tilde\pi^{\ast})\pi^{\ast}\pi(I-\tilde\pi-\tilde\pi^{\ast})=0
\end{eqnarray*}
by $\pi^{\ast}\pi=0$. Then
\begin{eqnarray*}
MM^{*}&=&|\a-\beta|^2\pi(I-\tilde\pi-\tilde\pi^{\ast})(I-\tilde\pi-\tilde\pi^{\ast})\pi^{\ast}\\
      &=& |\a-\beta|^2\pi(I-\tilde\pi-\tilde\pi^{\ast})\pi^{\ast}\\
      &=& |\a-\beta|^2(-\tilde\pi\pi^{\ast}-\tilde\pi\pi^{\ast}\pi^{\ast})\\
      &=& |\a-\beta|^2(-\tilde\pi\pi^{\ast}-\tilde\pi\pi^{\ast})\\
      &=&-2|\a-\beta|^2\pi\tilde\pi^{\ast}
\end{eqnarray*}
So $(M-M^\ast)^2=2|\a-\beta|^2\pi\tilde\pi^{\ast}$.

\end{proof}

\begin{prop}\label{prop:facn}
The nilpotent loop $n_{\a,1,N}$ defined in Table \ref{tab:upq:pneqq} can be factored as the product of some $g_{\a,\beta, L, K, W}$.
\end{prop}

\begin{proof}
By definition, $N = M - M^*$, where $M:(V\oplus sV)^\perp\to V$ is a rank $1$ linear map, and $V$ is maximal isotropic. Choosing a vector $v$ spanning $\im M$, and $w\in (V\oplus sV)^\perp$ such that $M(w) = -2v$ and $M(w^\perp)=0$, we are in the situation of Lemma \ref{fornupn2}. Recall $\pi :=\pi_{L,K^\perp}$ and $\tilde \pi :=\pi_{L,sL^\perp}$, and we   have 
\[
\pi\tilde\pi=\tilde\pi, ~~\tilde\pi\pi=\pi, ~~\pi^{\ast}\tilde\pi=\tilde\pi^{\ast}\pi=0, ~~(\pi\tilde\pi^{\ast})^{\ast}=\tilde\pi\pi^{\ast}=\pi\tilde\pi^{\ast}, 
\]
as well as $M=(\a-\beta)\pi(I-\tilde\pi-\tilde\pi^{\ast})$ and  $N^2=2|\a-\beta|^2\pi\tilde\pi^{\ast}$. 

Then $M^2=(M^\ast)^2=M^\ast M=0$ implies that $N^3 =-MM^\ast(M-M^\ast)=0$. The proof is complete by the following computation:  
\begin{eqnarray*}
& &g_{\beta,\a,L,K,W} \cdot g_{\a,\beta, L, sL, L^\perp \cap sL^\perp }\\
&=&( I + \frac{\alpha-\beta}{\lambda-\alpha} \, \pi + \frac{\bar\beta-\alpha}{\l-\bar\beta} \, \pi^\ast) \cdot ( I + \frac{\beta-\alpha}{\lambda-\beta} \, \tilde{\pi} + \frac{\a-\bar\beta}{\l-\a} \, \tilde{\pi}^\ast)\\
&=&I + \frac{\alpha-\beta}{\lambda-\alpha}  (\pi - \tilde{\pi} ) + \frac{\alpha-\bar\beta}{\lambda-\alpha}  (\tilde{\pi}^\ast - \pi^\ast ) + \frac{|\alpha-\beta|^2}{(\lambda-\alpha)^2} \pi  \tilde{\pi}^\ast\\
&=&I + \frac{\alpha-\beta}{\lambda-\alpha}  (\pi - \tilde{\pi} )-\frac{\alpha-\bar\beta}{\lambda-\alpha}(\pi - \tilde{\pi} )^{\ast}+\frac{|\alpha-\beta|^2}{(\lambda-\alpha)^2} \pi  \tilde{\pi}^\ast\\
&=&I+\frac{\a-\beta}{\l-\a}(\pi-\tilde\pi-\pi\tilde\pi^{\ast})+\frac{\a-\beta}{\l-\a}\pi\tilde\pi^{\ast}-\frac{\a-\bar\beta}{\l-\a}(\pi-\tilde\pi-\pi\tilde\pi^{\ast})^{\ast}\\
& &-\frac{\a-\bar\beta}{\l-\a}(\pi\tilde\pi^{\ast})^\ast+\frac{|\alpha-\beta|^2}{(\lambda-\alpha)^2} \pi  \tilde{\pi}^\ast\\
&=&I+\frac{M-M^\ast}{\l-\a}+\frac{|\alpha-\beta|^2}{(\lambda-\alpha)^2} \pi  \tilde{\pi}^\ast\\
&=&I+\frac{N}{\l-\a}+\frac{N^2}{2(\l-\a)^2}. 
\end{eqnarray*}
\end{proof}

Then by Proposition \ref{prop:facm}, \ref{prop:facn}, we obtain the main theorem for $\lupq$:

\begin{thm}\label{thm:fac4}
The rational loop group $\lupq$ is generated by the loops $p_{\alpha,L}$ and $g_{\alpha, \beta, L, K, W}$ given in Table \ref{tab:upq2}.
\end{thm}

\begin{rem}[\textbf{Generators of $\Un_{p,q}$ via projections}]
Restricting to the loop of extended real line, this main theorem implies an interesting and geometric construction of $\Un_{p,q}$ generalizing the classical unitary matrix construction by Hermitian orthogonal projections: 
The matrix group $\Un_{p,q}$ is generated by two kinds of elements via simple projections: $ e^{\ii \theta} \pi_L+\pi_{L^\perp}$ for any non-isotropic complex line $L$ and $ z \, \pi_{L}+ \dfrac{1}{\bar z} \, \pi_{K} + \pi_{W} $ for any two isotropic complex lines $L,K$ at general position, where $\C^n = L \oplus K \oplus W$, \; $W:=L^\perp \cap K^\perp$. 

In fact, the polar decomposition of $\Un_{p,q}$ implies a stronger result: $ r \, \pi_{L}+ \dfrac{1}{r} \, \pi_{sL} + \pi_{(L\oplus sL)^\perp} $ for real $r>0$ and isotropic $L$, together with $ e^{\ii \theta} \pi_L+\pi_{L^\perp}$ for non-isotropic $L$, are enough to generate 
$\Un_{p,q}$. 
\end{rem}

\begin{rem}
For the above theorem, we use $g_{\a,\beta, L,K,W}$ to replace $q_{\a,\beta,L}$ to decompose the nilpotent loops. In fact, the nilpotent loops generally can not be written as the product of pure $q_{\a,\beta,L}$. For example, in the $U_{1,1}$-case, assume
\[
m_{\a,1,N}=\prod_{i=1}^{l}q_{\a_i,\beta_i,L_i}.
\]
Compute derivatives on both sides of the equation to get
\[
\frac{-1}{(\l-\a)^2}N=\sum_{j=1}^{l}\prod_{i\leq j-1}q_{\a_i,\beta_i,L_i}(\frac{\a_j-\beta_j}{(\l-\beta_j)^2}\pi_{L_j}+\frac{\bar\beta_j-\bar\a_j}{(\l-
\bar\a_j)^2}\pi_{sL_j})\prod_{i\geq j+1}q_{\a_i,\beta_i,L_i}
\]
Multiplying the two sides of the equation by $(\l-\a)^2$, and evaluating at $\infty$ we get
\[
-N=\sum_{j=1}^{l}[(\a_j-\beta_j)\pi_{L_j}+(\bar\beta_j-\bar\a_j)\pi_{sL_j}].
\]
i.e.
\[
sN=-\sum_{j=1}^{l}[(\a_j-\beta_j)s\pi_{L_j}+(\bar\beta_j-\bar\a_j)\pi_{L_j}s].
\]
Then we have $0\neq {\tr}\, sN={\tr}\, (-\sum_{j=1}^{l}[(\a_j-\beta_j)s\pi_{L_j}+(\bar\beta_j-\bar\a_j)\pi_{L_j}s])=0$. It is a contradiction.
\end{rem}

\begin{rem}
So far we have stated the theorems for groups of rational loops that are normalized at $\infty$, as previously done in the literature.  All of them are true without this assumption, if we allow more general linear fractional transformations in the definition of these generators. This is due to the well-known fact that $\Aut(\C\pr^1)$ consists of such linear fractional (or M\"obius) transformations, also denoted $PGL_2 \C$. 
\end{rem}

\noindent
\textbf{Basic Examples}

\noindent Now we give some examples of the factorizations of nilpotent loops in the cases $\Un_{1,1}$ and $\Un_{1,2}$. They also gave the original hint for the general theorems.

We first discuss the case: $\Un_{1,1}$.
Let $\im N=\C\bpm 1&e^{\ii\gamma}\epm^t$, $\gamma\in\R$, then we can write $N=r\frac{1+e^{\ii\theta}}{1-e^{\ii\theta}}\bpm -1& e^{-\ii\gamma}\\-e^{\ii\gamma}&1\epm$, $r,\theta\in\R$. For $\a\in\R$, we have by Proposition \ref{prop:facm}
\begin{eqnarray*}
&&I+\frac{1}{(\l-\a)^2}\frac{1+e^{\ii\theta}}{2(1-e^{\ii\theta})}\bpm -1& e^{-\ii\gamma}\\-e^{\ii\gamma}&1\epm\\
&&=\left[\frac{\l-\a-\ii}{\l-\a}\frac{-1}{1- e^{-\ii\theta}}\bpm -1& e^{-\ii(\gamma+\theta)}\\-e^{\ii\gamma}&e^{-\ii\theta}\epm+\frac{\l-\a}{\l-\a+\ii}\frac{-1}{1- e^{\ii\theta}}\bpm -1& e^{-\ii\gamma}\\-e^{\ii(\gamma+\theta)}&e^{\ii\theta}\epm \right]\\
&&\left[\frac{\l-\a}{2(\l-\a-\ii)}\bpm 1&e^{-\ii\gamma}\\e^{\ii\gamma}& 1\epm+\frac{\l-\a+\ii}{2(\l-\a)}\bpm 1&-e^{-\ii\gamma}\\-e^{\ii\gamma}& 1\epm \right]
\end{eqnarray*}
\begin{eqnarray*}
&&I+\frac{1}{\l-\a}\frac{1+e^{\ii\theta}}{1-e^{\ii\theta}}\bpm -1& e^{-\ii\gamma}\\-e^{\ii\gamma}&1\epm\\
&&=\left[\frac{\l-\a+1}{\l-\a}\frac{-1}{1-e^{-\ii\theta}}\bpm -1& e^{-\ii(\gamma+\theta)}\\-e^{\ii\gamma}&e^{-\ii\theta}\epm+\frac{\l-\a}{\l-\a+1}\frac{-1}{1-e^{\ii\theta}}\bpm -1& e^{-\ii\gamma}\\-e^{\ii(\gamma+\theta)}&e^{\ii\theta}\epm \right]\\
&&\left[\frac{\l-\a}{2(\l-\a+1)}\bpm 1&e^{-\ii \gamma}\\e^{ \ii \gamma}& 1\epm+\frac{\l-\a+1}{2(\l-\a)}\bpm 1&-e^{-\ii \gamma}\\-e^{\ii \gamma}& 1\epm \right]\\
&&\left[\frac{\l-\a-\ii}{2(\l-\a)}\bpm 1&e^{-\ii \gamma}\\e^{\ii \gamma}& 1\epm+\frac{\l-\a}{2(\l-\a+\ii)}\bpm 1&-e^{-\ii \gamma}\\-e^{\ii \gamma}& 1\epm \right]\\
&&\left[\frac{\l-\a}{\l-\a-\ii}\frac{-1}{1-e^{-\ii\theta}}\bpm -1& e^{-\ii(\gamma+\theta)}\\-e^{\ii\gamma}&e^{-\ii\theta}\epm+\frac{\l-\a+\ii}{\l-\a}\frac{-1}{1-e^{\ii\theta}}\bpm -1& e^{-\ii\gamma}\\-e^{\ii(\gamma+\theta)}&e^{\ii\theta}\epm \right]
\end{eqnarray*}

For $\Un_{1,2}$, we give the factorization of the new nilpotent type $n_{\a,1,N}$. Assume $L=\C v$, $v=\bpm 1& \cos\theta & \sin\theta\epm^t$, $\theta\in\R$. We can choose $w=\bpm 0& -\sin\theta &\cos\theta\epm^t\in(L\oplus sL)^\perp$. Then by Lemma \ref{fornupn2}, we can get $\beta=\a-2$, $K=\C\bpm1& -\sin\theta & \cos\theta\epm^t$, 
\begin{eqnarray*}
 M=-2\bpm 1\\ \cos\theta \\ \sin\theta\epm \bpm 0& -\sin\theta & \cos\theta\epm, ~~~N=-2\bpm 0& -\sin\theta & \cos\theta\\  -\sin\theta &0 &1\\ \cos\theta &-1& 0\epm.
\end{eqnarray*}
Then by Proposition \ref{prop:facn}, we have
\begin{eqnarray*}
&& I+\frac{N}{\l-\a}+\frac{N^2}{2(\l-\a)^2}\\
&&=I+\frac{-2}{\l-\a}\bpm 0& -\sin\theta & \cos\theta\\  -\sin\theta &0 &1\\ \cos\theta &-1& 0\epm+\frac{2}{(\l-\a)^2}\bpm 1\\ \cos\theta \\ \sin\theta\epm \bpm 1& -\cos\theta & -\sin\theta\epm\\
&&=\left[ \frac{\l-\a+2}{\l-\a}\bpm 1\\ \cos\theta \\ \sin\theta\epm \bpm 1& \sin\theta &-\cos\theta\epm+\frac{\l-\a}{\l-\a+2}\bpm 1 \\ -\sin\theta \\ \cos\theta\epm \bpm 1 & -\cos\theta & -\sin\theta\epm + \right. \\
&& \left. \bpm 1 \\ \cos\theta-\sin\theta \\ \cos\theta+\sin\theta \epm \bpm -1 &  \cos\theta-\sin\theta & \cos\theta+\sin\theta\epm \right] \left[ \frac{\l-\a}{2(\l-\a+2)}\bpm 1 \\ \cos\theta \\ \sin\theta\epm \bpm 1 & \cos\theta & \sin\theta\epm \right. \\
&& \left. + \frac{\l-\a+2}{2(\l-\a)}\bpm -1 \\ \cos\theta \\ \sin\theta\epm \bpm -1 & \cos\theta & \sin\theta\epm+\bpm 0 \\ \sin\theta \\-\cos\theta\epm \bpm 0 & \sin\theta &-\cos\theta\epm \right].
\end{eqnarray*}

\section{Open problems}

We conclude the paper by formulating some open questions.\\

\indent (1). Here we just study generating theorems of the simplest rational loop groups. For loop groups with other noncompact real forms, for example, $\Un^{*}_{2n}$, even with the twisted forms, what are the generating theorems?\\

\indent (2). The proofs of our generating theorems are somewhat indirect: we first show generating theorems using projective and nilpotent loops, and then show that the nilpotent ones are generated by the projective ones. It would be desirable to have direct proofs, avoiding the usage of the nilpotent loops.\\

\indent (3). In the proof of our generating theorems, when the loop is written as the product of the generators, one needs to insert many fake singularities, leading to a long process of factorization. It would be interesting to find a set of generators that generate in such a way that fake singularities are forbidden. Maybe one needs to add more generators, for example in the $\Un_{p,q}$ case, some complexification of nilpotent loops: $I+\frac{1}{(\l-\a)(\l-\bar\a)}N$, $\a\in\C\setminus\R$, $N^2=0$, $N^*=-N$.\\

\indent (4). In \cite{Neil08}, the authors stated that the definition of rational loop groups depends on a choice of representation. All existing generating theorems are associated with the fundamental representation. It would be interesting to prove generating theorems associated with other representations. \\

\section*{Acknowledgement}

The authors would like to thank Professor ZHU Yongchang in the Hong Kong University of Science and Technology for many inspiring discussions about loop groups and Lie theory in general from the view point of algebra. The third author would like to express his  deepest gratitude for the support of HKUST during the project. He had also been supported by the NSF of China (Grant Nos. 10941002, 11001262), and the Starting Fund for Distinguished Young Scholars of Wuhan Institute of Physics and Mathematics (Grant No. O9S6031001).

\newcommand{\noopsort}[1]{}


\begin{thebibliography}{10}
\providecommand{\url}[1]{\texttt{#1}}
\providecommand{\urlprefix}{URL }
\providecommand{\eprint}[2][]{\url{#2}}


%
\bibitem{Neil08}
Donaldson, N., Fox, D. and Goertsches, O. \emph{Generators for rational loop groups}, Trans. Amer. Math. Soc. \textbf{363}(7) (2011), 3531-3552.

%
\bibitem{Goe13}
Goertsches, O., \emph{{G}enerating rational loop groups with noncompact reality conditions}, Math. Scand.  \textbf{113} (2013), 187-205.
%

\bibitem{haw68}
Hawkins, J. B.,Kammerer, W. J.,  \emph{A class of linear transformations which can be written as the product of projections}, Proc. Amer. Math. Soc.,\textbf{19}(1968), 739-745.

%
\bibitem{Lin152}
Lin, Z. C., Wang, G. and Wang, E., \emph{Dressing actions on proper definite affine spheres},  Asian J. Math.
\textbf{21} (2017), no. 2, 363--390.

%

\bibitem{Pressley1986}
Pressley A., Segal G.: \emph{Loop groups}; Oxford Mathematical Monographs; The
  Clarendon Press Oxford University Press, New York (1986); {O}xford Science
  Publications.

%
\bibitem{Te08}
Terng, C. L., \emph{{G}eometries and symmetries of soliton equations and integrable elliptic equations}, Surveys on geometry and integrable systems, 401--488, Adv. Stud. Pure. Math., 51, Math. Soc. Japan, Tokyo, 2008.
%
\bibitem{Ter00}
Terng, C. L. and Uhlenbeck, K., \emph{B\"acklund transformations and loop group actions},
Comm. Pure. Appl. Math., \textbf{53}(2000), 1-75.
%

\bibitem{Terng2005}
Terng C.L., Wang E.: \emph{Curved flats, exterior differential systems, and
  conservation laws}; in \emph{Complex, contact and symmetric manifolds};
  volume 234 of \emph{Progr. Math.}; pp. 235--254; Birkh\"auser Boston, Boston,
  MA (2005).

\bibitem{Ter081}
Terng, C. L. and Wang, E., \emph{Transformations of flat Lagrangian immersions and Egoroff
nets}, Asian J. Math. \textbf{12}(1): (2008), 99-119. 
%

\bibitem{Ter16}
Terng, C. L. and Wu, Z.W., \emph{Isotropic curve flows on $\R^{n,n+1}$}, arXiv:1608.0762.
%

\bibitem{Uhl89}
Uhlenbeck, K., \emph{Harmonic maps into Lie groups:classical solutions of the Chiral model}, J. Diff. Geom., \textbf{30}(1989), 1-50.
%

\bibitem{Wan06} 
Wang, E., \emph{Tzitz\'eica transformation is a dressing action}, J. Math. Phys. \textbf{47}(2006), no. 5, 053502, 13 pp. 
%

\bibitem{Zak79}
Zakharov, V.E. and Shabat, A.B., \emph{Integration of non-linear equations of mathematical
physics by the inverse scattering method, II}, Funct. Anal. Appl., \textbf{13} (1979), 166-174.

\end{thebibliography}
\end{document}